\documentclass[11pt]{amsart} \textwidth=14.5cm \oddsidemargin=1cm
\usepackage{amsmath,wasysym}
\usepackage{amsxtra}
\usepackage{amscd}
\usepackage{amsthm}
\usepackage{amsfonts}
\usepackage{amssymb}
\usepackage{eucal}
\usepackage{pmgraph}
\usepackage{stmaryrd}
\usepackage{hyperref,mathrsfs}
\usepackage{mathtools}
\usepackage{pictexwd,dcpic}

\newtheorem{thm}{Theorem}[section]
\newtheorem{cor}[thm]{Corollary}
\newtheorem{lem}[thm]{Lemma}
\newtheorem{prop}[thm]{Proposition}

\theoremstyle{definition}

\theoremstyle{remark}
\newtheorem{rem}[thm]{Remark}

\numberwithin{equation}{section}

\usepackage{color}

\begin{document}

\newcommand{\thmref}[1]{Theorem~\ref{#1}}
\newcommand{\secref}[1]{Section~\ref{#1}}
\newcommand{\lemref}[1]{Lemma~\ref{#1}}
\newcommand{\propref}[1]{Proposition~\ref{#1}}
\newcommand{\corref}[1]{Corollary~\ref{#1}}
\newcommand{\remref}[1]{Remark~\ref{#1}}
\newcommand{\eqnref}[1]{(\ref{#1})}
\newcommand{\exref}[1]{Example~\ref{#1}}

\newcommand{\nc}{\newcommand}
\nc{\Z}{{\mathbb Z}}
\nc{\C}{{\mathbb C}}
\nc{\N}{{\mathbb N}}
\nc{\F}{{\mf F}}
\nc{\Q}{\ol{Q}}
\nc{\la}{\lambda}
\nc{\ep}{\epsilon}
\nc{\h}{\mathfrak h}
\nc{\n}{\mf n}
\nc{\G}{{\mathfrak g}}
\nc{\ga}{{\gamma}}
\nc{\ella}{{\ell_{\lambda}}}
\nc{\DG}{\widetilde{\mathfrak g}}
\nc{\SG}{\overline{\mathfrak g}}
\nc{\D}{\mc D} \nc{\Li}{{\mc L}} \nc{\La}{\Lambda} \nc{\is}{{\mathbf
i}} \nc{\V}{\mf V} \nc{\bi}{\bibitem} \nc{\NS}{\mf N}
\nc{\dt}{\mathord{\hbox{${\frac{d}{d t}}$}}} \nc{\E}{\mc E}
\nc{\ba}{\tilde{\pa}} \nc{\half}{\frac{1}{2}} \nc{\mc}{\mathcal}
\nc{\mf}{\mathfrak} \nc{\hf}{\frac{1}{2}}
\nc{\hgl}{\widehat{\mathfrak{gl}}} \nc{\gl}{{\mathfrak{gl}}}
\nc{\hz}{\hf+\Z}
\nc{\dinfty}{{\infty\vert\infty}} \nc{\SLa}{\overline{\Lambda}}
\nc{\SF}{\overline{\mathfrak F}} \nc{\SP}{\overline{\mathcal P}}
\nc{\U}{\mathfrak u} \nc{\SU}{\overline{\mathfrak u}}
\nc{\ov}{\overline}
\nc{\wt}{\widetilde}
\nc{\osp}{\mf{osp}}
\nc{\spo}{\mf{spo}}
\nc{\hosp}{\widehat{\mf{osp}}}
\nc{\hspo}{\widehat{\mf{spo}}}
\nc{\I}{\mathbb{I}}
\nc{\X}{\mathbb{X}}
\nc{\Y}{\mathbb{Y}}
\nc{\hh}{\widehat{\mf{h}}}
\nc{\cc}{{\mathfrak c}}
\nc{\dd}{{\mathfrak d}}
\nc{\aaa}{{\mf A}}
\nc{\xx}{{\mf x}}
\nc{\wty}{\widetilde{\mathbb Y}}
\nc{\ovy}{\overline{\mathbb Y}}
\nc{\vep}{\bar{\epsilon}}
\nc{\vars}{{\phi}}
\nc{\bvars}{\underline{{\phi}}}
\nc{\bcdot}{{\boldsymbol{\cdot}}}
\nc{\Hom}{{\rm Hom}}
\nc{\Ext}{{\rm Ext}}
\nc{\ext}{{\rm ext}}
\nc{\Ninf}{{\N\cup\{\infty\}}}
\nc{\mcP}{{\mc{P}}}
\nc{\mcPne}{{\mc{P}^-}}
\nc{\Lalan}{{\Lambda^\la_n}}
\nc{\Lalanne}{{\Lambda^{\la,-}_n}}
\nc{\Opn}{{\mc{O}'_{n,\phi}}}
\nc{\Opk}{{\mc{O}'_{k,\phi}}}
\nc{\Lpn}{\Lambda'_{n,\phi}}
\nc{\Lpk}{\Lambda'_{k,\phi}}
\nc{\trp}{{\mf{tr}'}^n_k}

\advance\headheight by 2pt

\title[Projective modules in the parabolic category]
{Projective modules over classical Lie algebras of infinite rank in the parabolic category}

\author[Chen]{Chih-Whi Chen}
\address{Department of Mathematics, Uppsala University, Box 480,
	SE-75106, Uppsala, SWEDEN; Present address: School of Mathematical Sciences, Xiamen University, 
	Xiamen 361005, CHINA;
} \email{chihwhichen@xmu.edu.cn}

\author[Lam]{Ngau Lam}
\address{Department of Mathematics, National Cheng Kung University, Tainan, Taiwan 70101}
\email{nlam@mail.ncku.edu.tw}

\begin{abstract} We study the truncation functors and show the existence of projective cover with a finite Verma flag of each irreducible module in parabolic BGG category $\mc O$ over infinite rank Lie algebra of types $\mf{a,b,c,d}$. Moreover, $\mc O$ is a Koszul category. As a consequence, the corresponding parabolic BGG category $\overline{\mc O}$ over infinite rank Lie superalgebra of types $\mf{a,b,c,d}$ through the super duality is also a Koszul category.
\end{abstract}

 \maketitle

  \setcounter{tocdepth}{1}

\section{Introduction} The theory of Koszul algebra has powerful applications appearing in many areas of mathematics. In particular, the Koszul theory in representation theory has recently aroused some considerable interests.  In their celebrated paper \cite{BGS96}, Beilinson, Ginzburg and Soergel established the remarkable parabolic-singular duality by using the theory of Koszul algebras \cite{BGS96}. The Koszul theory for BGG category $O$ for complex semisimple Lie algebras has further extensively studied , see, e.g., \cite{Ba99}, \cite{BS3}, \cite{Ma07}, \cite{Ma10}, \cite{CM1}, \cite{CM2}, \cite{CM3}, \cite{MOS09}, \cite{SVV}.

 Along the line of the proof for Brundan's conjecture \cite{Br03}, the super duality conjecture was formulated and established by Cheng, Wang, Zhang and the second author (cf. \cite{CWZ08}, \cite{CW08}, \cite{CL10}, \cite{CLW11} and \cite{CLW12}). The super duality provides an equivalence of categories between the parabolic BGG categories $\overline{\mc{O}}$ for Lie superalgebras  and $\mc{O}$ for Lie algebras of infinite rank of types $\mf{a,b,c,d}$. The super duality plays a crucial role  in solving Brundan's conjecture in \cite{CLW15} and establishing the irreducible character of Lie superalgebra of type $B$ in \cite{BW} and type $D$ in \cite{B}. Thus the parabolic BGG categories $\mc O$ over infinite rank Lie algebras of types $\mf{a,b,c,d}$ has been paid attention. There is another proof for Brundan's conjecture given by Brundan, Losev and Webster \cite{BLW}.

 The Koszulity of categories of finite-dimensional modules over Lie superalgebras of  type $A$ had been proved by Brundan and Stroppel, see \cite{BS2, BS4}.  Recently, Brundan, Losev and Webster \cite{BLW} established the Koszulity for BGG category $O$ for Lie superalgebra of type $A$ in its full generality based on the uniqueness of tensor product categorification. However, the problems of Koszulity for the BGG categories $O$ of Lie superalgebras beyond type $A$ are still unclear.

The main result in this article shows the Koszulity holds for the parabolic BGG categories $\mc O$ for infinite rank Lie algebras  of types $\mf{a,b,c,d}$.  The existence of projective cover of each irreducible module in $\mc O$ is proved by knowing the existence of projective covers for finite rank cases and the equivalence of categories consisting of modules equipped with finite Verma flags over Lie algebras of various ranks. The idea for showing the Koszulity for the $\mc O$ is described as follows. By knowing that the endomorphism algebra $R_{n,\la}$ of the minimal projective generator of each block of $\mc O_n$ over Lie algebra of finite rank is Koszul \cite{Ba99}, we show that $R_{k,\la}$ with the Koszul grading can be regarded as a graded subalgebra of $R_{n,\la}$ for $k<n$ and then their direct limit is also Koszul. We also show that $\mc O$ has a Kazhdan-Lusztig theory in the sense of \cite{CPS92}. As a consequence, the corresponding parabolic BGG categories $\overline{\mc O}$ over infinite rank Lie superalgebras of types $\mf{a,b,c,d}$ through the super duality is also a Koszul category.

 The paper is organized as follows. In Section \ref{sec::Liealgebras}, the infinite rank Lie algebras $\mf g^{\xx}$ of types $\xx = \mf{a,b,c}, \mf d$ and their categories, and the properties of truncation functors, are recalled.
In Section \ref{Sect::3}, the equivalence of categories consisting of modules equipped with finite Verma flags for Lie algebras of various ranks is established. The existence of projective cover of each irreducible module in $\mc O$ is proved in this section.  We also show that $\mc O$ has a Kazhdan-Lusztig theory.

 The Section \ref{sect::EmbedingofRings} is devoted to the proof of the Koszulity for $\mc O$.  The relations between endomorphism algebras $R_{n,\la}$ of minimal projective generators of blocks of $\mc O_n$  at various ranks are studied and we show that $R_{k,\la}$ with the Koszul grading can be regarded as a graded subalgebra of $R_{n,\la}$ for $k<n$. The relations of the extension algebra of the direct sum of irreducibles at various ranks are also studied in Section \ref{sect::EmbedingofRings}. The parabolic BGG categories $\overline{\mc O}$ over infinite rank Lie superalgebras of types $\mf{a,b,c,d}$ are Koszul categories obtained in Section \ref{subsection::SO}. In Section \ref{Set::final}, we establish the Koszulity for the dual category along the same line without proof.

\vskip 0.5cm
\noindent{\bf Notations.} Throughout the paper the symbols $\Z$, $\N$, and $\Z_+$ stand for the sets of all,
positive and non-negative integers, respectively. All vector spaces, algebras, tensor
products, et cetera, are over the field of complex numbers $\C$.

\vskip 0.5cm
\noindent{\bf Notes added.} Before the present paper is completed,  some main results of type A in our paper were obtained by Christopher Leonard in \cite{Le}. Also, some main results in our paper have been obtained by Kevin Coulembier and Ivan Penkov.

\vskip 0.5cm
\noindent{\bf Acknowledgment.} The authors thank Kevin Coulembier for valuable discussions. The first author is supported by Vergstiftelsen.
The second author was partially supported by MoST grant 104-2115-M-006 -015 -MY3 of Taiwan.
\bigskip

\section{Lie algebras of infinite and finite ranks}\label{sec::Liealgebras}
 In this section, we first recall the infinite rank Lie algebras $\G^\xx$ and the parabolic BGG categories $\mc {O}$ of $\G^\xx$-modules, where $\xx$ denotes one of the four types $\mf{a,b,c,d}$. Then we recall the truncation functors which relate to the parabolic categories for $\G^\xx$ and finite-dimensional Lie algebras. We refer the reader to \cite[Section 2 and 3]{CL10} for type $\mf{a}$ and \cite[Section 2 and 3]{CLW11} for types $\mf{b,c,d}$ for their details (see also \cite[Section 6.1 and 6.2]{CW12}). We fix $m\in\Z_+$ in this article.

\subsection{Lie algebras}
Consider the free abelian group with basis $\{\epsilon_{-m},\ldots,\epsilon_{-1}\}\cup\{\epsilon_{r}\vert
r\in\N\}$, with a symmetric bilinear form $(\cdot|\cdot)$
given by
\begin{align*}
(\epsilon_r|\epsilon_s)=\delta_{rs}, \qquad r,s \in
\{-m,\ldots,-1\} \cup \N.
\end{align*}
We set
\begin{align}\label{alpha:beta}
&
 \quad\alpha_{j}
:=\epsilon_{j}-\epsilon_{j+1},\quad & -m\le j\le -2,\\
&\beta_{\times}:=\epsilon_{-1}-\epsilon_{1},\quad
\beta_{r}
:=\epsilon_{r}-\epsilon_{r+1},\quad & r\in\N.\nonumber
\end{align}

For $\xx =\mf{a,b,c,d}$, we denote by $\mf{k}^\xx$ the contragredient Lie algebras
whose Dynkin diagrams \makebox(23,0){$\oval(20,12)$}\makebox(-20,8){$\mf{k}^\xx$} together with certain
distinguished sets of simple roots $\Pi(\mf{k^x})$ listed as follows ($m\ge 2$): \vspace{.3cm}

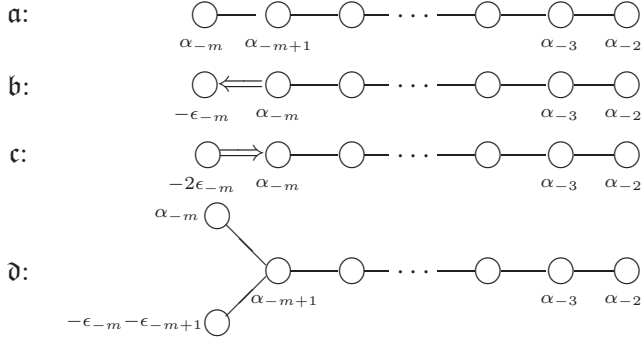
\begin{center}
\hskip -3cm \setlength{\unitlength}{0.16in}
\begin{picture}(24,2)
\put(8,2){\makebox(0,0)[c]{$\bigcirc$}}
\put(10.4,2){\makebox(0,0)[c]{$\bigcirc$}}
\put(14.85,2){\makebox(0,0)[c]{$\bigcirc$}}
\put(17.25,2){\makebox(0,0)[c]{$\bigcirc$}}
\put(19.4,2){\makebox(0,0)[c]{$\bigcirc$}}
\put(5.6,2){\makebox(0,0)[c]{$\bigcirc$}}
\put(8.4,2){\line(1,0){1.55}} \put(10.82,2){\line(1,0){0.8}}
\put(13.2,2){\line(1,0){1.2}} \put(15.28,2){\line(1,0){1.45}}
\put(6,2){\line(1,0){1.25}}
\put(17.7,2){\line(1,0){1.25}}
\put(12.5,1.95){\makebox(0,0)[c]{$\cdots$}}
\put(-.5,2){\makebox(0,0)[c]{$\mf{a}$:}}
\put(5.5,1){\makebox(0,0)[c]{\tiny$\alpha_{-m}$}}
\put(8,1){\makebox(0,0)[c]{\tiny$\alpha_{-m+1}$}}
\put(17.2,1){\makebox(0,0)[c]{\tiny$\alpha_{-3}$}}
\put(19.3,1){\makebox(0,0)[c]{\tiny$\alpha_{-2}$}}
\end{picture}
\end{center}
\begin{center}
\hskip -3cm \setlength{\unitlength}{0.16in}
\begin{picture}(24,2)
\put(8,2){\makebox(0,0)[c]{$\bigcirc$}}
\put(10.4,2){\makebox(0,0)[c]{$\bigcirc$}}
\put(14.85,2){\makebox(0,0)[c]{$\bigcirc$}}
\put(17.25,2){\makebox(0,0)[c]{$\bigcirc$}}
\put(19.4,2){\makebox(0,0)[c]{$\bigcirc$}}
\put(5.6,2){\makebox(0,0)[c]{$\bigcirc$}}
\put(8.4,2){\line(1,0){1.55}} \put(10.82,2){\line(1,0){0.8}}
\put(13.2,2){\line(1,0){1.2}} \put(15.28,2){\line(1,0){1.45}}
\put(17.7,2){\line(1,0){1.25}}
\put(6,1.8){$\Longleftarrow$}
\put(12.5,1.95){\makebox(0,0)[c]{$\cdots$}}
\put(-.5,2){\makebox(0,0)[c]{$\mf{b}$:}}
\put(5.5,1){\makebox(0,0)[c]{\tiny$-\epsilon_{-m}$}}
\put(8,1){\makebox(0,0)[c]{\tiny$\alpha_{-m}$}}
\put(17.2,1){\makebox(0,0)[c]{\tiny$\alpha_{-3}$}}
\put(19.3,1){\makebox(0,0)[c]{\tiny$\alpha_{-2}$}}
\end{picture}
\end{center}
\begin{center}
\hskip -3cm \setlength{\unitlength}{0.16in}
\begin{picture}(24,2)
\put(5.7,2){\makebox(0,0)[c]{$\bigcirc$}}
\put(8,2){\makebox(0,0)[c]{$\bigcirc$}}
\put(10.4,2){\makebox(0,0)[c]{$\bigcirc$}}
\put(14.85,2){\makebox(0,0)[c]{$\bigcirc$}}
\put(17.25,2){\makebox(0,0)[c]{$\bigcirc$}}
\put(19.4,2){\makebox(0,0)[c]{$\bigcirc$}}
\put(6.8,2){\makebox(0,0)[c]{$\Longrightarrow$}}
\put(8.4,2){\line(1,0){1.55}} \put(10.82,2){\line(1,0){0.8}}
\put(13.2,2){\line(1,0){1.2}} \put(15.28,2){\line(1,0){1.45}}
\put(17.7,2){\line(1,0){1.25}}
\put(12.5,1.95){\makebox(0,0)[c]{$\cdots$}}
\put(-.5,2){\makebox(0,0)[c]{$\mf{c}$:}}
\put(5.5,1){\makebox(0,0)[c]{\tiny$-2\epsilon_{-m}$}}
\put(8,1){\makebox(0,0)[c]{\tiny$\alpha_{-m}$}}
\put(17.2,1){\makebox(0,0)[c]{\tiny$\alpha_{-3}$}}
\put(19.3,1){\makebox(0,0)[c]{\tiny$\alpha_{-2}$}}
\end{picture}
\end{center}
\begin{center}
\hskip -3cm \setlength{\unitlength}{0.16in}
\begin{picture}(24,3.5)
\put(8,2){\makebox(0,0)[c]{$\bigcirc$}}
\put(10.4,2){\makebox(0,0)[c]{$\bigcirc$}}
\put(14.85,2){\makebox(0,0)[c]{$\bigcirc$}}
\put(17.25,2){\makebox(0,0)[c]{$\bigcirc$}}
\put(19.4,2){\makebox(0,0)[c]{$\bigcirc$}}
\put(6,3.8){\makebox(0,0)[c]{$\bigcirc$}}
\put(6,.3){\makebox(0,0)[c]{$\bigcirc$}}
\put(8.4,2){\line(1,0){1.55}} \put(10.82,2){\line(1,0){0.8}}
\put(13.2,2){\line(1,0){1.2}} \put(15.28,2){\line(1,0){1.45}}
\put(17.7,2){\line(1,0){1.25}}
\put(7.6,2.2){\line(-1,1){1.3}}
\put(7.6,1.8){\line(-1,-1){1.3}}
\put(12.5,1.95){\makebox(0,0)[c]{$\cdots$}}
\put(-.5,2){\makebox(0,0)[c]{$\mf{d}$:}}
\put(3.3,0.3){\makebox(0,0)[c]{\tiny${-}\epsilon_{-m}{-}\epsilon_{-m+1}$}}
\put(4.7,3.8){\makebox(0,0)[c]{\tiny$\alpha_{-m}$}}
\put(8.2,1){\makebox(0,0)[c]{\tiny$\alpha_{-m+1}$}}
\put(17.2,1){\makebox(0,0)[c]{\tiny$\alpha_{-3}$}}
\put(19.3,1){\makebox(0,0)[c]{\tiny$\alpha_{-2}$}}
\end{picture}
\end{center}

We have the following identification of Lie algebras: $\mf{k}^\mf{a}= \mf{gl}(m)$, $\mf{k}^\mf{b}=\mf{so}(2m+1)$, $\mf{k}^\mf{c}=\mf{sp}(2m)$ for $m\ge 1$ and $\mf{k}^\mf{d} = \mf{so}(2m)$ for $m \ge 2$, respectively.

For $n\in\N$, let
\makebox(23,0){$\oval(20,12)$}\makebox(-20,8){$\mf{T}_n$}
denote the following Dynkin diagram:
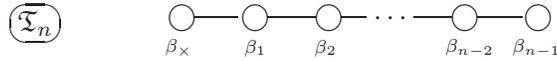
\begin{center}
\hskip -3cm \setlength{\unitlength}{0.16in}
\begin{picture}(24,3)
\put(8,2){\makebox(0,0)[c]{$\bigcirc$}}
\put(10.4,2){\makebox(0,0)[c]{$\bigcirc$}}
\put(14.85,2){\makebox(0,0)[c]{$\bigcirc$}}
\put(17.25,2){\makebox(0,0)[c]{$\bigcirc$}}
\put(5.6,2){\makebox(0,0)[c]{$\bigcirc$}}
\put(8.4,2){\line(1,0){1.55}} \put(10.82,2){\line(1,0){0.8}}
\put(13.2,2){\line(1,0){1.2}} \put(15.28,2){\line(1,0){1.45}}
\put(6,2){\line(1,0){1.4}}
\put(12.5,1.95){\makebox(0,0)[c]{$\cdots$}}
\put(0,1.2){{\ovalBox(1.6,1.2){$\mf{T}_n$}}}
\put(5.5,1){\makebox(0,0)[c]{\tiny$\beta_{\times}$}}
\put(8,1){\makebox(0,0)[c]{\tiny$\beta_{1}$}}
\put(10.3,1){\makebox(0,0)[c]{\tiny$\beta_{2}$}}
\put(15,1){\makebox(0,0)[c]{\tiny$\beta_{n-2}$}}
\put(17.2,1){\makebox(0,0)[c]{\tiny$\beta_{n-1}$}}
\end{picture}
\end{center}

The Lie algebras associated with these Dynkin diagrams are
$\gl(n+1)$. In the limit $n\to\infty$, the associated algebra are direct limits of
these Lie algebras.

Any of the {\em head} diagrams
\makebox(23,0){$\oval(20,11)$}\makebox(-20,8){$\mf{k}^\xx$} may be
connected with the {\em tail} diagram
\makebox(23,0){$\oval(20,12)$}\makebox(-20,8){$\mf{T}_n$} to
produce the following Dynkin diagram ($n\in\N\cup\{\infty\}$), denoted by \makebox(23,0){$\oval(20,13)$}\makebox(-20,8){$\mf{\G}_n^\xx$}:
\begin{equation}\label{Dynkin:combined}
\hskip -3cm \setlength{\unitlength}{0.16in}
\begin{picture}(24,1)
\put(15,0.5){\makebox(0,0)[c]{{\ovalBox(1.6,1.2){$\mf{k}^\xx$}}}}
\put(15.8,0.5){\line(1,0){1.85}}
\put(18.5,0.5){\makebox(0,0)[c]{{\ovalBox(1.6,1.2){${\mf{T}}_n$}}}}
\end{picture}
\end{equation}
We will use the same notation
\makebox(23,0){$\oval(20,13)$}\makebox(-20,8){$\mf{\G}_n^\xx$} to denote the diagrams of all
the degenerate cases. We recall  for the case $m=1$ defined below \cite[Section 2.5]{CLW11}:

\vspace{.3cm}

\begin{center}
\hskip -3cm \setlength{\unitlength}{0.16in}
\begin{picture}(24,2)
\put(8,2){\makebox(0,0)[c]{$\bigcirc$}}
\put(10.4,2){\makebox(0,0)[c]{$\bigcirc$}}
\put(14.85,2){\makebox(0,0)[c]{$\bigcirc$}}
\put(17.25,2){\makebox(0,0)[c]{$\bigcirc$}}
\put(19.4,2){\makebox(0,0)[c]{$\bigcirc$}}
\put(5.6,2){\makebox(0,0)[c]{$\bigcirc$}}
\put(8.4,2){\line(1,0){1.55}} \put(10.82,2){\line(1,0){0.8}}
\put(13.2,2){\line(1,0){1.2}} \put(15.28,2){\line(1,0){1.45}}
\put(6,2){\line(1,0){1.25}}
\put(17.7,2){\line(1,0){1.25}}
\put(12.5,1.95){\makebox(0,0)[c]{$\cdots$}}
\put(5.5,1){\makebox(0,0)[c]{\tiny$\beta_{\times}$}}
\put(8,1){\makebox(0,0)[c]{\tiny$\beta_{1}$}}
\put(10.4,1){\makebox(0,0)[c]{\tiny$\beta_{1}$}}
\put(17.2,1){\makebox(0,0)[c]{\tiny$\beta_{n-2}$}}
\put(19.3,1){\makebox(0,0)[c]{\tiny$\beta_{n-1}$}}
\put(-1,1.2){{\ovalBox(1.6,1.2){$\G^{\mf a}_n$}}}
\end{picture}
\end{center}
\begin{center}
\hskip -3cm \setlength{\unitlength}{0.16in}
\begin{picture}(24,2)
\put(8,2){\makebox(0,0)[c]{$\bigcirc$}}
\put(10.4,2){\makebox(0,0)[c]{$\bigcirc$}}
\put(14.85,2){\makebox(0,0)[c]{$\bigcirc$}}
\put(17.25,2){\makebox(0,0)[c]{$\bigcirc$}}
\put(19.4,2){\makebox(0,0)[c]{$\bigcirc$}}
\put(5.6,2){\makebox(0,0)[c]{$\bigcirc$}}
\put(8.4,2){\line(1,0){1.55}} \put(10.82,2){\line(1,0){0.8}}
\put(13.2,2){\line(1,0){1.2}} \put(15.28,2){\line(1,0){1.45}}
\put(17.7,2){\line(1,0){1.25}}
\put(6,1.8){$\Longleftarrow$}
\put(12.5,1.95){\makebox(0,0)[c]{$\cdots$}}
\put(5.5,1){\makebox(0,0)[c]{\tiny$-\epsilon_{-1}$}}
\put(8,1){\makebox(0,0)[c]{\tiny$\beta_{\times}$}}
\put(10.4,1){\makebox(0,0)[c]{\tiny$\beta_{1}$}}
\put(17.2,1){\makebox(0,0)[c]{\tiny$\beta_{n-2}$}}
\put(19.3,1){\makebox(0,0)[c]{\tiny$\beta_{n-1}$}}
\put(-1,1.2){{\ovalBox(1.6,1.2){$\G^{\mf b}_n$}}}
\end{picture}
\end{center}
\begin{center}
\hskip -3cm \setlength{\unitlength}{0.16in}
\begin{picture}(24,2)
\put(5.7,2){\makebox(0,0)[c]{$\bigcirc$}}
\put(8,2){\makebox(0,0)[c]{$\bigcirc$}}
\put(10.4,2){\makebox(0,0)[c]{$\bigcirc$}}
\put(14.85,2){\makebox(0,0)[c]{$\bigcirc$}}
\put(17.25,2){\makebox(0,0)[c]{$\bigcirc$}}
\put(19.4,2){\makebox(0,0)[c]{$\bigcirc$}}
\put(6.8,2){\makebox(0,0)[c]{$\Longrightarrow$}}
\put(8.4,2){\line(1,0){1.55}} \put(10.82,2){\line(1,0){0.8}}
\put(13.2,2){\line(1,0){1.2}} \put(15.28,2){\line(1,0){1.45}}
\put(17.7,2){\line(1,0){1.25}}
\put(12.5,1.95){\makebox(0,0)[c]{$\cdots$}}
\put(5.5,1){\makebox(0,0)[c]{\tiny$-2\epsilon_{-1}$}}
\put(8,1){\makebox(0,0)[c]{\tiny$\beta_{\times}$}}
\put(10.4,1){\makebox(0,0)[c]{\tiny$\beta_{1}$}}
\put(17.2,1){\makebox(0,0)[c]{\tiny$\beta_{n-2}$}}
\put(19.3,1){\makebox(0,0)[c]{\tiny$\beta_{n-1}$}}
\put(-1,1.2){{\ovalBox(1.6,1.2){$\G^{\mf c}_n$}}}
\end{picture}
\end{center}
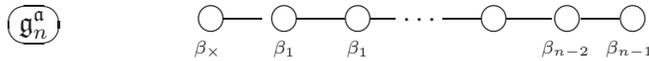
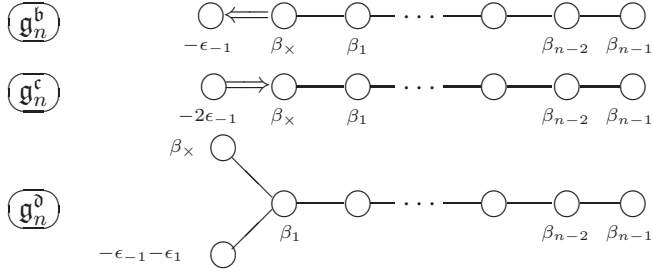
\begin{center}
\hskip -3cm \setlength{\unitlength}{0.16in}
\begin{picture}(24,3.5)
\put(8,2){\makebox(0,0)[c]{$\bigcirc$}}
\put(10.4,2){\makebox(0,0)[c]{$\bigcirc$}}
\put(14.85,2){\makebox(0,0)[c]{$\bigcirc$}}
\put(17.25,2){\makebox(0,0)[c]{$\bigcirc$}}
\put(19.4,2){\makebox(0,0)[c]{$\bigcirc$}}
\put(6,3.8){\makebox(0,0)[c]{$\bigcirc$}}
\put(6,.3){\makebox(0,0)[c]{$\bigcirc$}}
\put(8.4,2){\line(1,0){1.55}} \put(10.82,2){\line(1,0){0.8}}
\put(13.2,2){\line(1,0){1.2}} \put(15.28,2){\line(1,0){1.45}}
\put(17.7,2){\line(1,0){1.25}}
\put(7.6,2.2){\line(-1,1){1.3}}
\put(7.6,1.8){\line(-1,-1){1.3}}
\put(12.5,1.95){\makebox(0,0)[c]{$\cdots$}}
\put(3.3,0.3){\makebox(0,0)[c]{\tiny${-}\epsilon_{-1}{-}\epsilon_{1}$}}
\put(4.7,3.8){\makebox(0,0)[c]{\tiny$\beta_{\times}$}}
\put(8.2,1){\makebox(0,0)[c]{\tiny$\beta_{1}$}}
\put(17.2,1){\makebox(0,0)[c]{\tiny$\beta_{n-2}$}}
\put(19.3,1){\makebox(0,0)[c]{\tiny$\beta_{n-1}$}}
\put(-1,1.2){{\ovalBox(1.6,1.2){$\G^{\mf d}_n$}}}
\end{picture}
\end{center}

We now recall the explicit matrix realization of the corresponding Lie algebras. For $m\in\N$, let ${\I}_m$ be the following totally ordered set:
\begin{align*}
\cdots <\ov{2}
<\ov{1}<\underbrace{\ov{-1}<\ov{-2}
<\cdots<\ov{-m}}_m<\ov{0}<\underbrace{-m<\cdots<-1}_m
<1<2<\cdots
\end{align*}
We let ${V}_{m}$ be the infinite dimensional space over $\C$ with
ordered basis $\{v_i|i\in {\I}_m\}$. With respect to this basis a linear map on
${V}_m$ may be identified with a complex matrix $(a_{rs})_{r,s\in {\mathbb{I}}_m}$.
The Lie algebra $\gl({V}_m)$ is the Lie subalgebra of linear transformations on
${V}_m$ consisting of $(a_{rs})$ with $a_{rs}=0$ for all but finitely many $a_{rs}$'s.
Denote by $E_{rs}\in\gl({V}_m)$ the elementary matrix with $1$ at the $r$th row and
$s$th column and zero elsewhere.

For $n\in\N\cup\{\infty\}$, we define the subset ${\mathbb I}_m^+(n)$ of the set ${\mathbb I}_m$ as follows:
\begin{align*}
\I^+_m(n) :=\{-m,\ldots,-1\}\cup\{i\in \N\,|\, i< n+1\} \quad \hbox{and}\quad \I^+_m:=\I^+_m(\infty).
\end{align*}
For $\xx =\mf a,  \mf c, \mf b, \mf d$, $m\in\N$ and $n\in \N\cup\{\infty\}$, the Lie algebra $\mathcal{G}_n^\xx$ corresponding
to the Dynkin diagram \makebox(23,0){$\oval(20,13)$}\makebox(-20,8){$\mf{\G}_n^\xx$} can be identified as a subalgebra of $\gl({V}_m)$ such that the standard Cartan subalgebra with a basis and corresponding dual basis $\{\epsilon_i\}$,
$$
\begin{cases} \{E_i:=E_{ii}\,|\,\ i\in{\I}^+_m(n)\},\quad &\text{for }\xx  =\mf a,\\
\{E_i:=E_{ii}-E_{\ov{i}\ov{i}}\,|\,i\in{\I}^+_m(n)\},\quad &\text{for }\xx  = \mf {b,c,d}
\end{cases}
$$
and all positive roots vectors are the nonzero multiple of the elements in the following set
$$
\begin{cases} \{E_{ij}\,|\,i< j,\ i,j\in{\I}^+_m(n)\},\ &\text{for }\xx  =\mf a,\\
 \{E_{ij}-E_{\ov{j}\ov{i}}\,|\,i< j,\ i,j\in{\I}^+_m(n)\}
  \cup\{E_{\ov{i}j}-E_{\ov{j}i}\,|\,i< j,\ i,j\in{\I}^+_m(n)\}\\
 \ \ \cup \{E_{\ov{i}\ov{0}}-E_{\ov{0}i}\,|\ i\in{\I}^+_m(n)\},\ &\text{for }
\xx  = \mf b,\\
\{E_{ij}-E_{\ov{j}\ov{i}}\,|\,i< j,\,i,j\in{\I}^+_m(n)\}
 \cup \{E_{\ov{i}j}+E_{\ov{j}i}\,|\,i\le j,\, i\le j\in{\I}^+_m(n)\},\ &\text{for }
\xx  = \mf c,\\
\{E_{ij}-E_{\ov{j}\ov{i}}\,|\,i< j,\ i,j\in{\I}^+_m(n)\}
 \cup\{E_{\ov{i}j}-E_{\ov{j}i}\,|\,i< j,\ i,j\in{\I}^+_m(n)\},\ &\text{for }
\xx  = \mf d.
\end{cases}
$$
Let $\mf{\G}_n^\xx$ be the central extension
of $\mathcal{G}_n^\xx$ by the one-dimensional center $\C K$ determined by the $2$-cocycle
\begin{align*}
\tau(A,B):=\text{tr}([J,A]B),\qquad A, B\in\mathcal{G}_n^\xx,
\end{align*}
where
$$
J:=\begin{cases} -\sum_{r\geq 1}E_{rr},\quad &\text{for }\xx  =\mf a,\\
E_{\ov{0}\ov{0}}+\sum_{r\leq \bar 1}E_{\ov{r}\ov{r}},\quad &\text{for }\xx  = \mf {b,c,d},
\end{cases}
$$
and $\text{tr}$ denotes the trace. Note that the central extensions are trivial. Indeed, there is an isomorphism $\iota$ from $\mathcal{G}_n^\xx\oplus\C K$ to $\mf{\G}_n^\xx$ (cf. \cite[Section 2.5]{CLW11}) defined by
\begin{equation}\label{iota}
\iota(A)= A+\text{tr}(JA) K, \quad \hbox{for } A\in  \mathcal{G}_n^\xx\qquad\hbox{and} \qquad \iota(K)= K.
\end{equation}

For $n\in \N\cup\{\infty\}$, let $\mathcal{H}^\xx_n$ and $\h^\xx_n$ denote the standard Cartan subalgebras of $\mathcal{G}^\xx_n\oplus\C K$ and $\G^\xx_n$ with bases $\{K,{E}_{r}\, ( r\in{\I}^+_m(n))\}$ and  $\{K,{E}_{r}\, ( r\in{\I}^+_m(n))\}$, respectively.
For each positive simple root $\alpha$ in the diagram \eqnref{Dynkin:combined},  the reflection associated to $\alpha$ on $\mathcal{H}^\xx_n$ is defined by:
\begin{align*}
\sigma_{\alpha}(h):= h -\alpha(h)h_\alpha,\quad \hbox{for all } h\in \mathcal{H}^\xx_n\cap \mathcal{G}^\xx_n \quad  \text{and}\quad
\sigma_{\alpha}(K) :=K,
\end{align*}
where $h_\alpha$ is the coroot of $\alpha$ in $\mathcal{G}^\xx_n$.
We define the reflection on $(\h_n^\xx)^*$, also denoted by $\sigma_{\alpha}$, in term of $\sigma_{\alpha}$ defined above by
\begin{align*}
\sigma_{\alpha}(\mu)(h):=\mu\big(\iota \sigma_{\alpha}^{-1}\iota^{-1}(h)\big),\qquad  \hbox{ for all }\,\,\mu\in (\h_n^\xx)^*  \,\hbox{ and }\, h\in \h^\xx_n.
\end{align*}

The Weyl group $W^\xx_n$ of $\G^\xx_n$ is defined to be the subgroup of $Aut((\h^\xx_n)^*)$ generated by all $\sigma_{\alpha}$ (cf. \cite[Section 2.3]{HLT12}).

Let $\rho_n^\xx\in(\h^\xx_n)^*$ be defined by
\begin{equation*}
\begin{aligned}
& \rho_n^\xx({E}_{j}):=
\begin{cases}
-m-\hf-j+\delta_j, & \text{for $x=\mf{b}$, $j\in \I^+_m(n) $}, \\
-m-1-j+\delta_j, & \text{for $x=\mf{c}$, $j\in \I^+_m(n) $}, \\
 -m-j+\delta_j, & \text{for $x=\mf{a},\mf{d}$, $j\in \I^+_m(n) $},
\end{cases}
\end{aligned}
\quad  \text{and}\quad
\rho^\xx_n(K) :=0,
\end{equation*}
where $\delta_j$ is defined by
\begin{align}\label{delta}
 \delta_j:=
\begin{cases}
1, & \text{for $j\in \N $}, \\
0, & \text{for $j\notin \N $}.
\end{cases}
\end{align}
The dot action of $W^\xx_n$ on $(\h^\xx_n)^*$ is given by:
$$
w\cdot\mu:= w \left(\mu+\rho_n^\xx\right)-\rho_n^\xx, \text{ for } w\in W^\xx_n\,\,\text{and} \,\, \mu \in (\h^\xx_n)^*.
$$

\subsection{Parabolic BGG category}\label{sec:para} In this subsection, we recall the parabolic categories.
For $n\ge 1$, we set $Y_n^\xx$ to be a fixed subset of the simple roots in the \makebox(23,0){$\oval(20,13)$}\makebox(-20,8){$\mf{\G}_n^\xx$} such that $\beta_\times\notin Y_n^\xx$ and $Y_n^\xx\supseteq \{\beta_i \, |\,1\le i<n\}$.
Let $\mf{l_n^\xx}$ denote the Levi subalgebra of $\G_n^\xx$ corresponding to the subset $Y_n^\xx$.
The Borel subalgebra of $\G_n^\xx$ spanned by the central element $K$ and 
the upper triangular matrices in $\G_n^\xx$ is denoted by $\mf{b}_n^\xx$. Let $\mf{p}_n^\xx :=\mf{l}_n^\xx +\mf{b}_n^\xx$ be the corresponding parabolic subalgebra with
nilradical $\mf{u}_n^\xx$ and opposite nilradical $\mf{u}_{n-}^\xx$.

\vskip 0.3cm
\noindent{\bf In the remainder of the paper we shall drop the superscript
$\xx$, which denotes a fixed type among
$\mf{a, b,c,d}$.}

\vskip 0.3cm

For $n\in \N\cup\{\infty\}$, let $\La_0$ be an element in $(\h_n)^*$ defined by letting
 \begin{equation*}
\langle\La_0,K\rangle = \Lambda_0(K):= 1
\qquad
\hbox {and}
\qquad
\langle \La_0, {E}_{r}\rangle = \Lambda_0(E_r):=0, \,\,\, \hbox {for all $r$.}
\end{equation*}
We define the set of {\em integral} weights to be
\begin{align*}
X_n:=\{\la:=\sum_{i=-m}^{n}\la_i\epsilon_i + d\La_0\in \h_n^*\, |\,& \la(\alpha^{\vee})\in\Z,\, \forall \alpha\in \Pi(\G_n);\, \la_i  \in\Z, \forall i\ge 1;\\
&\la_j=0 \,\, \hbox{for } j\gg 0\, \hbox{ and }\, d\in\hf\Z\},
\end{align*}
and let
\begin{equation*}\mcP^-_n:=\{\la\in X_n\,|\, \la(\alpha^{\vee})\in\Z_+,\, \forall\alpha\in Y_n\}\,\,\,\text{and}\,\,\,
\mathcal{P}_n:=\{\la\in \mc P_n^-\,|\, \la(E_j)\in\Z_+,\, \forall j\ge 1\},\label{weight:Im}
\end{equation*}
where $\Pi(\G_n):=\Pi(\mf{k})\cup\{\beta_j\,|\, 1\le j<n\}\cup\{\beta_\times\}$ is a set of simple root of $\G_n$ and $\alpha^{\vee}:=\iota(h_\alpha)$ is the coroot of $\alpha$. Note that $\mathcal{P}_n= \mc{P}^-_n$ for $n=\infty$ but $\mathcal{P}_n\subsetneq \mcP^-_n$ for $n\not=\infty$. For $n,k \in \N \cup \{\infty\} $ with $n \geq k$, we identify $\mc{P}_k$ (resp. $X_k$) as a subset of $\mc{P}_n$ (resp. $X_n$) by setting $\la(E_j)=0$ for $j>k$ and $\la\in \mc{P}_k$ (resp. $X_k$). An integral weight $\mu\in X_n$ is said to be {\em dominant} if $(\mu+\rho_n)(\alpha^\vee)\ge 0$ for all $\alpha \in \Pi(\G_n)$.

For $n\in\Ninf$ and a semisimple $\h_n$-module $M$, let $M_\mu$ denote the $\mu$-weight space of $M$ for $\mu\in \h_n^*$.
For $\mu\in \mcP^-_n$, let $L(\mf{l}_n,\mu)$ denote the irreducible highest weight
 $\mf{l}_n$-module of highest weight $\mu$. We extend
$L(\mf{l}_n,\mu)$ to a $\mf{p}_n$-module by letting $\mf{u}_n$ act
trivially.  Define as usual the parabolic Verma module
$\Delta_n(\mu)$ and its irreducible quotient $L_n(\mu)$ over $\G_n$ by
\begin{align*}
\Delta_n(\mu):=\text{Ind}_{\mf{p}_n}^{\G_n}L(\mf{l}_n,\mu)\qquad\text{and}\qquad
\Delta_n(\mu) \twoheadrightarrow L_n(\mu).
\end{align*}
For $n\in \N\cup\{\infty\}$, let $\mc{O}_n^-$ be the full subcategory of the category of $\G_n$-modules  such that every $M\in \mc{O}_n^-$ is a semisimple
${\h}_n$-module with finite-dimensional weight spaces $M_\gamma$, for $\gamma\in X_n$,
satisfying
\begin{itemize}
\item[(i)] ${M}$ over ${\mf{l}_n}$ decomposes into a direct sum of
$L({\mf{l}_n},\mu)$ for $\mu\in \mc{P}^-_n$.
\item[(ii)] There exist finitely many weights
$\la^1,\la^2,\ldots,\la^k\in\mc{P}^-_n$ (depending on ${M}$) such that
if $\gamma$ is a weight in ${M}$, then
$\gamma\in\la^i-\sum_{\alpha\in\Pi(\G_n)}\Z_+\alpha$, for some $i$.
\end{itemize}
For $d\in \hf\Z$, let $\mc{O}_n^-(d)$ denote the full subcategory $\mc{O}_n^-$ consisting of modules $M$ satisfying $$Kv=d\La_0(K)v=dv,\qquad \,\, \text{for all }\,\, v\in M.$$
Let $\mc{O}_n$ be the full subcategory of $\mc{O}_n^-$ such that  every  $M\in\mc{O}_n$ decomposes into a direct sum of
$L({\mf{l}_n},\mu)$ for $\mu\in \mathcal{P}_n$  and let $\mc{O}_n(d):=\mc{O}_n\cap\mc{O}_n^-(d)$.  Note that $\mc{O}_n$ are abelian categories and $\mc{O}_n=\mc{O}_n^-$ for $n=\infty$ but $\mc{O}_n\subsetneq \mc{O}_n^-$ for $n\not=\infty$. It is clear that $\mc{O}_n^-=\bigoplus_{d\in\hf\Z}\mc{O}_n^-(d)$ and $\mc{O}_n=\bigoplus_{d\in\hf\Z}\mc{O}_n(d)$. The  categories $\mc{O}$ and $\mc{O}(d)$ are indeed  subcategories of categories of $\mc{O}$ and $\mc{O}(d)$ defined in \cite[Section 6.2.3]{CW08} consisting of integral weights modules, respectively. We define $(\mc{O}_n^-)_f$ as the full subcategory of $\mc{O}_n^-$ consisting of finitely
generated $\G_n$-modules. Also, $(\mc{O}_n)_f$, $\mc{O}_n(d)_f$, etc. are defined in the similar way. Note that, $\mc{O}_n^-=(\mc{O}_n)_f^-$, for $n\in\N$, are the usual parabolic BGG categories and hence $\mc{O}_n^-$ have enough projectives and injectives.

For $n\in \N\cup\{\infty\}$ and  $M \in \mc O^-_n$, we define the {\em dual} module of $M$ as follows:
\begin{equation}
\label{Eq::def::duality} M^{\vee} := \bigoplus_{\mu \in X_n}M^*_{\mu},
\end{equation}
where $M^\vee$ carries the standard $\mf g_n$-action through $\tau$, where $\tau : \mf g_n \longrightarrow \mf g_n$ is the transpose map (see, for example, \cite[Section 0.5]{Hum08}). Then the functor $\vee$, called the {\em duality} functor, is exact from $\mc O^-_n$ (resp. $\mc O_n$) to itself sending each irreducible module to itself such that $\vee\circ\vee$ is naturally isomorphic to the identity functor (see, for example, \cite[Section  3.2]{Hum08}).  For $n\in \N\cup\{\infty\}$ and $\mu\in \mcP^-_n$, $\nabla_n(\mu):=\Delta_n(\mu)^\vee$ is called a {\em costandard} module.

We will simply drop $n$ for $n=\infty$. For example, $\mc{O}=\mc{O}_{\infty}$,
$X =X_{\infty}$, $\mathcal{P}=\mathcal{P}_\infty$, $\mf{l}= \mf{l}_{\infty}$, $\G=\G_{\infty}$, $L(\la)=L_{\infty}(\la)$ and $\Delta(\la)=\Delta_{\infty}(\la)$, etc.

\begin{lem}\emph{(}\cite[Corollary 6.8]{CW12}\emph{)}\label{2.1}
For $n\in \N\cup\{\infty\}$ and $\mu\in \mathcal{P}_n$,
$\Delta_n(\mu)$, $\nabla_n(\mu)$ and $L_n(\mu)$ lie in $\mc O_n$. Furthermore, $\{L_n(\ga)|~\ga\in \mathcal{P}_n\}$ is a complete list of non-isomorphic irreducible modules in $\mc O_n$.
\end{lem}

\subsection{Truncation functors} In this subsection, we recall the truncation functors and their properties on the cohomology groups and the extension groups.
Recall that $\mc{P}_k$ (resp. $X_k$), for $0\le k< n\le\infty$, is regarded  as  a subset of $\mc{P}_n$ (resp. $X_n$) by defining $\la(E_j)=0$ for all $k<j$.  The {\em truncation functor}
\begin{equation}\label{tr}
\mf{tr}^{n}_k:\mc{O}_n \longrightarrow\mc{O}_k
\end{equation}
is defined by $\mf{tr}^{n}_k(M):=\bigoplus_{\mu\in X_k}M_\mu$ for $M\in \mc{O}_n$  (see \cite[Section 6.2.5]{CW12}). Note that truncation functors are exact functors  commuting with the duality functor. We refer to \cite[Section 6.2.5]{CW12} for more details.

\begin{lem}\label{lem:trunc} \emph{(\cite[Proposition 6.9]{CW12})}
For $0\le k< n \le \infty$, $Z=L,\Delta, \nabla$ and $\la \in\mathcal{P}_n$,
\begin{align*}\mf{tr}_k^n\big{(}Z_n(\la)\big{)} =
\begin{cases}
Z_k (\la),&\quad\text{if }
\la \in  \mc{P}_k;\\
0,&\quad\text{otherwise}.
\end{cases} \end{align*}
\end{lem}

For ${M},{N}\in{\mc{O}_n}$ and $i\in \Z_+$, the $i$-th extension group ${\rm
Ext}_{{\mc{O}_n}}^i({M},{N})$ is defined in the sense of Yoneda (see,
for example,~\cite[Chapter VII]{Mit65}). In particular, we have ${\rm Ext}_{{\mc{O}_n}}^0({M},{N})={\rm
Hom}_{{\G_n}}({M},{N})$.

For a Lie algebra $u$ and a $u$-module $N$, let
$H^i(u;{N})$ denote the $i$-th restricted (in the sense of \cite{Liu} for ${\rm dim}\, u=\infty$) $u$-cohomology group of $N$ (\cite[\S 6.4.1]{CW12}). The following proposition is well-known when $n\in \N$ (see, e.g., \cite[\S7 Theorem 2]{RW82}). For $n=\infty$, it is a part of \cite[Theorem 6.34]{CW12}). Note that there is no non-trivial extension between  two irreducible $\mf{l}_n$-modules $L(\mf{l}_n,\mu)$ and $L(\mf{l}_n,\la)$ for $\mu, \la\in \mathcal{P}_n$ and $n\in \N\cup\{\infty\}$.

\begin{prop}\label{Prop::relative ext} For $n\in \N \cup \{\infty\} $, $\mu\in \mc{P}_n$ and $N\in \mc{O}_n$, we have
\begin{align*}
{\rm Ext}^i_{\mc O_{n}}
(\Delta_n(\mu),N) \cong
{\rm Hom}_{{\mf{l}}_n}\big{(}
L({\mf{l}}_n,\mu),H^i({\mf{u}}_n;{N})\big{)},\qquad\text{for $i=0,1$}.
\end{align*}
\end{prop}

The first part of the following proposition is an analogue of \cite[Corollary 4.5]{CL10}.
\begin{prop}\label{CoHomologiesMatch}  Let $0\le k< n \le \infty$ and $M \in \mc{O}_n$.
\begin{itemize}
  \item[(i)]
 For $i \in \Z_+$, we have
$$
\mf{tr}^n_k(H^i({\mf{u}}_n,M))=  H^i({\mf{u}}_k, \mf{tr}_k^n(M)).
$$
In particular,  for $Z=L,\Delta,\nabla$ and $\gamma \in \mc{P}_n$, we have
\begin{align*} \mf{tr}^n_k(H^i({\mf{u}}_n,Z_n(\gamma))) =  \left\{ \begin{array}{ll}  H^i({\mf{u}}_k,Z_k(\gamma)),& \quad \text{  if $\gamma \in \mc{P}_k$};
\\ 0, & \quad \text{ otherwise}.
 \end{array} \right.
 \end{align*}
  \item[(ii)] For $M \in \mc{O}_n$ and $\mu\in \mc{P}_k$, we have an isomorphism
  \begin{align*}
   \mf{tr}_k^n: {\rm Ext}^i_{\mc O_{n}}(\Delta_n(\mu),M) \longrightarrow{\rm Ext}^i_{\mc O_{k}}(\Delta_k(\mu),\mf{tr}_k^n(M)), \,\,\,\, \text{ for $i=0,1$.}
  \end{align*}
\end{itemize}
\end{prop}
\begin{proof}
	The proof is similar to the proof of \cite[Corollary 4.5]{CL10}. We first observe that the truncation of complex computing the (restricted, for $n=\infty$) ${\mf{u}}_n$-cohomology of any module $M\in {\mc O_{n}}$ is the complex computing the ${\mf{u}}_k$-cohomology of the truncated module $\mf{tr}_k^n(M)$. Therefore we prove the first part of (i). The second part of (i) is a direct consequence of \lemref{lem:trunc} and (ii) follows from \propref{Prop::relative ext} and (i).
\end{proof}

The following lemma is very useful. The proof is similar to the proof of \cite[Lemma 5.6]{CL10}.

\begin{lem} \label{lem::ShortInduceLong} Let $0\le k< n \le \infty$ and $M \in \mc{O}_n$.
Assume that $N \in \mc{O}_n$ and there is a short exact sequence
\begin{equation}\label{short:exact}
0\longrightarrow M_1\stackrel{}{\longrightarrow} M_2 {\longrightarrow} M_3\longrightarrow
0
\end{equation}
of $\mf g_n$-modules in $\mc{O}_n$. Let $N':=\mf{tr}_k^n(N)$ and $M_i':= \mf{tr}_k^n(M_i)$, for $ i =1,2, 3$. Then we have the following
commutative diagram of abelian groups with exact rows:
\begin{eqnarray*}
\CD
 0  @>>> {\rm
Hom}_{{\mc{O}_n}}({M}_3,N) @>>>{\rm Hom}_{\mc{O}_n}(M_2,N) @>>> {\rm Hom}_{\mc{O}_n}(M_1,N) \\
 @. @VV\mf{tr}_k^n V @VV\mf{tr}_k^n V @VV\mf{tr}_k^n V \\
 0  @>>> {\rm
Hom}_{\mc{O}_k}(M_3',N') @>>> {\rm
Hom}_{\mc{O}_k}(M_2',N') @>>> {\rm
Hom}_{\mc{O}_k}(M_1',N')  \\
 \endCD
\end{eqnarray*}
\begin{eqnarray*}
\hskip 0cm\CD
@>{\partial}>> {\rm Ext}^1_{\mc{O}_n}
(M_3,N) @>>>{\rm Ext}^1_{\mc{O}_n}
(M_2,N) @>>> {\rm Ext}^1_{\mc{O}_n}
(M_1,N)  \\
 @. @VV\mf{tr}_k^n V @VV\mf{tr}_k^n V @VV\mf{tr}_k^n V \\
 @>\partial>> {\rm Ext}^1_{\mc{O}_k}(M_3',N')
 @>>> {\rm Ext}^1_{\mc{O}_k}(M_2',N')
@>>> {\rm Ext}^1_{\mc{O}_k}(M_1',N')   \\
 \endCD
\end{eqnarray*}
where $\partial$ are the connecting homomorphisms.
\end{lem}

\subsection{Graded algebras}\label{subsection::KRandGM}  We set up some notations for graded algebras and recall the definition of  Koszul algebra. We refer the reader to \cite{BGS96, BS1, BS2} for more details.

For an abelian category $\mc{A}$ and $i\ge 0$,  let ${\rm Ext}^i_{\mc{A}}(M ,N)$ denote the $i$-th extension group of $N$ by $M$ in  $\mc{A}$ in the sense of Yoneda (see, for example,~\cite[Chapter VII]{Mit65}) and let ${\rm Ext}^\bullet_{\mc{A}}(M ,M):=\bigoplus_{i\ge 0}{\rm Ext}^i_{\mc{A}}(M ,M)$ denote the extension algebra for $M\in\mc{A}$

An algebra always means an associative (but not necessarily unital) $\C$-algebra. The opposite algebra of an algebra $A$ is denoted by $A^{\rm op}$. An algebra $A$ is called {\em locally unital} (see, e.g., \cite[Section 5]{BS1}) if there is a system $\{e_\lambda \,|\, \lambda\in \Lambda\}$ of mutually orthogonal idempotents such that
$$A =\sum_{\alpha, \beta\in \Lambda} e_\alpha A e_\beta.$$
Note that $A$ is an unital algebra with identity element $\sum_{\alpha\in \Lambda} e_\alpha$ if $A$ is finite-dimensional over $\C$. For our applications, we only consider positively graded locally unital algebra $A= \bigoplus_{i\ge 0}A_i$ with $A_0 = \bigoplus_{\alpha \in \Lambda}\mathbb{C}e_{\alpha}$. For a locally unital algebra $A$, an $A$-module $M$ always means a left $A$-module $M$ in the usual sense
such that
$$M =\bigoplus_{\alpha\in \Lambda} e_\alpha M.$$
For a  locally unital (resp. graded) algebra $A$, we let $A$-Mod (resp. $A$-mod) denote the category of all (resp. graded) left $A$-modules and let Hom$_{A}(M, N)$ (resp. hom$_{A}(M, N)$) denote the set of $A$-homomorphisms (resp. homogenous $A$-homomorphisms of degree zero) from $M$ to $N$ for $M, N\in A$-Mod (resp. $A$-mod).  For any $i\in \Z$ and $M \in A$-mod with $M =\bigoplus_{j\in \Z}M_j$, we write $M\langle i\rangle$ for the same module but with new grading defined by $M\langle i\rangle_j:= M_{j-i}$ for all $j\in \Z$.
For $i\ge 0$, the group ${\rm Ext}^i_{A\text{-Mod}}(M ,N)$ (resp. ${\rm ext}_{A\text{-mod}}^i(M ,N)$) is denoted by ${\rm Ext}^i_A(M ,N)$ (resp. ${\rm ext}_A^i(M ,N)$). Let $A$-Mof and $A$-mof denote the subcategories of finitely generated $A$-modules in $A$-Mod and $A$-mod, respectively.  Finally, we define Mod-$A$, mod-$A$, Mof-$A$ and mof-$A$ to be the categories of right $A$-modules in a similar manner. For $i\ge 0$, the $i$-th extensions of $N$ by $M$ in the Mod-$A$ (resp. mod-$A$) is denoted by ${\rm Ext}^i_{\text{-}A}(M ,N)$ (resp. ${\rm ext}_{\text{-}A}^i(M ,N)$).

Let $A_{>0}:= \bigoplus_{i\geq 1}A_k$ for any positively graded $\C$-algebra $A = \bigoplus_{i\geq 0}A_i$.
Recall that a positive graded algebra $A$ is called a {\em Koszul algebra} (see, e.g., \cite[Definition 1.1.2]{BGS96}) if $A_0$ is a semi-simple algebra and the graded left $A$-module $A_0\cong A/A_{> 0}$ admits a graded projective resolution in $A$-mod
\begin{align*}
\cdots \longrightarrow P^2\longrightarrow P^1 \longrightarrow P^0 \longrightarrow A_0 \longrightarrow 0,
\end{align*}
such that $P^i$ is generated by its $i$-th component, that is, $P^i = AP^i_i$, for each $i\in \Z_+$. Such a grading on the algebra $A$ is referred to as a {\em Koszul grading}.

Let $\mc{A}$ be an abelian $\C$-linear category such that every simple object $L_\mu$ has a projective cover $P_\mu$. Let $\{L_\mu\,|\, \mu\in \Lambda\}$ be a complete set of pairwise non-isomorphic simple objects. Assume that $\mc{A}$ is a full subcategory of an abelian $\C$-linear category $\mc{C}$ such that $\bigoplus_{\mu\in \Lambda}P_\mu$ is in $\mc{C}$. For $\gamma \in\Lambda$, let $e_{\gamma} \in {\rm End}_{\mc{C}}(\bigoplus_{\mu\in \Lambda}P_\mu)$ be the natural projection onto $P_\gamma$. The category $\mc{A}$ is called {\em Koszul} if the locally unital algebra $(\bigoplus_{\mu, \gamma\in  \Lambda}e_{\mu}{\rm End}_{\mc{C}}(\bigoplus_{\mu'\in \Lambda}P_{\mu'})e_{\gamma})^{\rm op}$ is Koszul. Note that an algebra $A$ is Koszul if and only if $A^{\rm op}$ is Koszul (see, for example, \cite[proposition 2.2.1]{BGS96}).

\section{Truncation functors and projective modules} \label{Sect::3}
In this section, we show that there are a projective cover and an injective envelope of each irreducible module in $\mc O_n$ for $n\in \Ninf$. First, we consider the full subcategory $\mc O^{\Delta}_n$ of $\mc{O}_n$ consisting of modules equipped with finite Verma flags and show that certain subcategory of $ \mc O^{\Delta}_{n}$ is equivalent to the category $\mc O^{\Delta}_k$ for $n\ge k$ (See \thmref{Thm::Equivlanece} below). Using the equivalence of categories and the existence of projective covers in $\mc O_n$ for $n\in \N$, we are able to show the existence of projective covers in $\mc O$.

For  $n\in \N\cup\{\infty\}$ and $M \in \mc O_n$, let
$$[M : L_n(\mu)]:={\rm dim}\, \Hom_{\G_n}(P_n(\mu), M), \quad \text{for all $\mu\in \mc{P}_n$.}$$  Note that $[M : L_n(\mu)]$ is equal to the supremum of the multiplicity of $ L_n(\mu)$ in the subquotients of $M$ with finite composition series and $[M : L_n(\mu)]$ is finite number since every module in $\mc O_n$ has finite dimensional weight spaces.
 We say that $M\in \mc O_n^-$ (resp. $\mc O_n$) has a finite {\em Verma flag} if there is filtration
\begin{align*}
0={M}_{r}\subset {M}_{r -1} \subset\cdots \subset M_0 =M,
\end{align*}
satisfying $M_{i}/M_{i+1} \cong \Delta_n(\gamma^i)$ and $\gamma^i\in \mcP^-_n$ (resp. $\mc{P}_n$), for $0\leq i\leq r-1$. Let $\mc{O}^{-,{\Delta}}_n$ (resp. $\mc{O}^{\Delta}_n$) denote the full subcategory of $\mc{O}_{n}$ consisting of all $\G_n$-modules possessing finite Verma flags. For $M \in \mc{O}^{-,{\Delta}}_n$ (resp. $\mc O_n^{\Delta}$) and $\mu \in  \mcP^-_n$, let $(M : \Delta_n(\mu))$ denote the multiplicity of Verma module $\Delta_n(\mu)$ occurring in a Verma
flag of $M$. Using the similar arguments as in the finite rank
cases, $(M : \Delta(\mu))$ is independent of the choices of the Verma flags.

For $n\in \N\cup\{\infty\}$ and $\ga\in \mc{P}_n$, let $P_n(\ga)$ (resp. $I_n(\mu)$) be the projective cover (resp. injective envelope) of $L_n(\gamma)$ in $\mc O_n$ if it exists.

\subsection{Modules equipped with Verma flags}\label{Verma flag}
In this subsection, we consider the category $\mc O^{\Delta}_n$ and the full subcategory $ \mc O^{\Delta}_{n, k}$ of $\mc O^{\Delta}_n$  for $n\ge k$.  We establish the crucial equivalence of the category  $ \mc O^{\Delta}_{n, k}$ and the category $\mc O^{\Delta}_k$ induced by truncation functor for $n\ge k$ (See \thmref{Thm::Equivlanece} below).

Let $\mc{O}^{\Delta}_{n,k}$ denote the full subcategory of  $\mc O^{\Delta}_{n}$ of all $\mf g_n$-modules $M$ with a finite Verma flag
\begin{align*}
0={M}_{r}\subset {M}_{r -1} \subset\cdots \subset M_0 =M,
\end{align*} satisfying
$M_{i}/M_{i+1} \cong \Delta_n(\gamma^i)$ and $\gamma^i\in\mc{P}_k$, for $0\leq i\leq r-1$.
The following proposition is a  strengthened version of the second part of \propref{CoHomologiesMatch}.

\begin{prop} \label{stExt}  For $0\le k< n \le \infty$, $M\in \mc O^{\Delta}_{n,k}$ and $N \in \mc O_{n}$, we have
\begin{itemize}
  \item[(i)] $ \mf{tr}_k^n: {\rm Hom}_{\mc O_{n}}(M,N) \longrightarrow {\rm Hom}_{\mc O_{k}}(\mf{tr}_k^n(M),\mf{tr}_k^n(N))$ is an isomorphism.
  \item[(ii)] $\mf{tr}_k^n: {\rm Ext}^1_{\mc O_{n}}(M,N) \longrightarrow {\rm Ext}^1_{\mc O_{k}}(\mf{tr}_k^n(M),\mf{tr}_k^n(N))$ is a monomorphism.
\end{itemize}
\end{prop}

\begin{proof} We shall proceed with our proof by induction on the length of the Verma flag of $M$. If the length of the the Verma flag of $M$ is $1$, then $M$ is isomorphic to a Verma module and hence the proposition holds  by \propref{CoHomologiesMatch}(ii). Assume it is true for all module $M\in\mc{O}_{n,k}^\Delta$ with the length of the Verma flag less than $r-1$.
Let  $0={M}_{r}\subset {M}_{r -1} \subset\cdots \subset M_0 =M$ be a finite filtration of $M$
 satisfying
$M_{i}/M_{i+1} \cong \Delta_n(\gamma^i)$ and $\gamma^i\in\mc{P}_k$, for $0\leq i\leq r-1$. We apply \lemref{lem::ShortInduceLong} to the module $N$ and the short exact sequence
$$
0\longrightarrow M_1 \longrightarrow M \longrightarrow M/M_1 \longrightarrow 0.
$$
Note that
the first and third vertical homomorphisms are isomorphisms by induction, the fourth vertical homomorphism is also an isomorphism by \propref{CoHomologiesMatch}(ii), and sixth vertical homomorphism is a monomorphism by induction. By Five Lemma and Four Lemma, the second vertical homomorphism is an isomorphism and the fifth vertical homomorphism is a monomorphism. The proof is completed.
\end{proof}

In view of \lemref{lem:trunc}, we may conclude that the restriction of truncation functor $\mf{tr}^n_k$ induces an exact functor from $\mc O^{\Delta}_{n,k}$ to $\mc O^{\Delta}_k$, which is also denoted by $\mf{tr}_{k}^n$.

\begin{thm} \label{Thm::Equivlanece} For $0\le k< n \le \infty$, the truncation functor
 $$ \mf{tr}_{k}^n :  \mc O^{\Delta}_{n, k}\longrightarrow  \mc O^{\Delta}_k$$ is an equivalence.
\end{thm}

\begin{proof}
By \propref{stExt}, $ \mf{tr}_{k}^n$  is a fully faithful  functor from $\mc O^{\Delta}_{n, k}$ to $\mc O^{\Delta}_k$.
We shall use induction on the length of the Verma flag of a given module in $\mc O^{\Delta}_k$ to show that $\mf{tr}_{k}^n$ is essentially surjective. Let $M\in \mc{O}_k^\Delta$ be a Verma module, then by \lemref{lem:trunc} M is isomorphic to an image of a Verma module under the functor $\mf{tr}_{k}^n$ . 
 Assume it is true for all module $M\in\mc{O}_k^\Delta$ with the length of the Verma flag less than $r-1$.
Let $M\in\mc{O}_k^\Delta$ with finite filtration of $\mf g_k$-modules:
\begin{align*}
0={M}_{r}\subset {M}_{r -1} \subset\cdots \subset M_0 =M,
\end{align*} satisfying
$M_{i}/M_{i+1} \cong \Delta_k(\gamma^i)$ and $\gamma^i\in\mc{P}_k$, for $0\leq i\leq r-1$.
By induction hypothesis, there are $N_1, N_2 \in \mc O_{n,k}^{\Delta}$ such that $\mf{tr}_k^n(N_1)\cong M_1$ and $\mf{tr}_k^n(N_2) \cong M/M_1$. Now $M$ is isomorphic to $\mf{tr}_k^n(N)$ for some $N\in\mc{O}_{n,k}^\Delta$ by \propref{CoHomologiesMatch}\,(ii).
\end{proof}


\subsection{Projective modules} In this subsection, we show that there are a projective cover and an injective envelope of each irreducible module in $\mc O_n$ for $n\in \Ninf$.

For $n\in \N\cup\{\infty\}$, there is a partial ordering on  $\mu, \la\in \mcP^-_n$ by declaring $\mu\le \la$, if
$\la-\mu$ can be written as a nonnegative integral linear combination of simple roots of $\G_n$.

\begin{lem}\label{le}
For $n\in \N$,  $\gamma\in \mcP^-_n\backslash \mc{P}_n$ and $\mu \in \mc{P}_n$, we have $\gamma\nleqslant\mu$.
\end{lem}

\begin{proof}  Since $\gamma\in \mcP^-_n\backslash \mc{P}_n$, we have $\gamma(E_n)<0$ and hence $(\mu-\gamma)(E_n)>0$. On other hand, the value of each nonnegative integral linear combination of simple roots on $E_n$ is a non-positive number. Therefore $\gamma\nleqslant\mu$.
\end{proof}

\begin{prop}\label{Oproj-finite} For $n\in \N$ and $\mu \in \mc{P}_n$, there are a projective cover $P_n(\mu)$ and an injective envelope $I_n(\mu)$ of $L_n(\mu)$ in $\mc O_n$. Moreover, we have
\begin{equation}\label{BGGre}
 (P_n(\la): \Delta_n(\mu)) = [\Delta_n(\mu): L_n(\la)], \quad \text{for $\la, \mu \in \mc{P}_n$.}
\end{equation}
Futhermore, we have
$$
\mf{tr}_{k}^{n}(P_{n}(\mu))\cong P_k(\mu),\,\,\,\, \text{and}\,\,\,\,  \mf{tr}_{k}^{n}(I_{n}(\mu))\cong I_k(\mu) \quad \text {for $0\le k< n < \infty$ and $\mu\in \mc{P}_k$}.
$$
\end{prop}

\begin{proof}
For $\mu\in \mc{P}_n$, let ${P}^-_n(\mu)$ be a projective cover of $L_n(\mu)$ in $\mc{O}^-_n$. Then $P^-_n(\mu)$ has a Verma flag $0={M}_{r}\subset {M}_{r -1} \subset\cdots \subset M_0 =M$ satisfying
$M_{i}/M_{i+1} \cong \Delta_n(\gamma^i)$ with $\gamma^i\in \mc{P}_n^-$, for each $0\leq i\leq r-1$. Applying the argument of the proof of \cite[Proposition 3.7(a)]{Hum08} repeatedly and using \lemref{le}, there is a number $s$ such that $\gamma^i\in \mc{P}_n$ for $0\le i\le s$ and $\gamma^i\notin \mc{P}_k$ for $i>s$. Let $P_n(\mu)=P_n(\mu)/M_{s+1}$. It is clear that $P_n(\mu)\in \mc{O}_n$ by \lemref{2.1} and it has a simple top. Since $\Hom_{\G_n}(M_{s+1}, N)=0$ for all $N\in \mc{O}_n$, $P_n(\mu)$ is  a projective cover of $L_n(\mu)$ in $\mc O_n$. There is also an injective envelope $I_n(\mu):=P_n(\mu)^\vee$ of $L_n(\mu)$ in $\mc O_n$.

The equality \eqref{BGGre} follows from the the BGG reciprocity for $\mc{O}_n^-$ and the construction of $P_n(\mu)$. By \thmref{Thm::Equivlanece}, there exists $M\in \mc{O}_{n,k}^{\Delta}$ such that $\mf{tr}_{k}^n(M) \cong P_k(\mu)$. By \propref{stExt} and ${\rm Ext}^1_{\mc O_{k}}(P_k(\mu),\mf{tr}_k^n(N))=0$ for all $N\in \mc O_n$, we have ${\rm Ext}^1_{\mc O_{n}}(M,N)=0$ for all $N\in \mc O_n$ and hence $M$ is a projective module in $\mc O_n$. By \eqref{BGGre} and Lemma \ref{lem:trunc}, it is easy to see that the multiplicity of each Verma module appearing in the Verma flags of $M$ and $P_n(\mu)$ are equal. Also $M$ and $P_n(\mu)$ are projective modules with the same quotient isomorphic to $L_n(\mu)$. Since $P_n(\mu)$ is a projective cover of $L_n(\mu)$,  we have $M\cong P_n(\mu)$ and $\mf{tr}_{k}^{n}(P_{n}(\mu))\cong\mf{tr}_{k}^{n}(M)\cong P_k(\mu)$.
\end{proof}

\begin{thm} \label{Oproj} For each $\mu \in \mc{P}$, there is a projective cover $P(\mu):=P_{\infty}(\mu)$ of $L(\mu)$ in $\mc O$. Moreover, $P(\mu)$ has a finite Verma flag and
\begin{equation}\label{BGGre-inf}
 (P(\la): \Delta(\mu)) = [\Delta(\mu): L(\la)], \quad \text{for $\la, \mu \in \mc{P}$.}
\end{equation}
\end{thm}

\begin{proof} Let $n=\infty$ and $\mu\in \mc{P}$. We can choose a sufficient large $k$ such that $\mu \in \mc{P}_k$. Using the same argument as in the last part of the proof of \propref{Oproj-finite}, we have a projective module $M\in \mc O$ such that $\mf{tr}_{k}^n(M) \cong P_k(\mu)$. We have $\mf{tr}_{k'}^n(M) \cong P_{k'}(\mu)$ for $k\le k'<\infty$ since $\mf{tr}_{k}^{k'}(P_{k'}(\mu))\cong P_k(\mu)$,  $\mf{tr}_{k}^{n}=\mf{tr}_{k}^{k'} \circ \mf{tr}_{k'}^{n}$ and $\mf{tr}_{k}^{k'}$ is an equivalence of categories from $\mc O_{k',k}^{\Delta}$ to $\mc O_{k}^{\Delta}$. Certainly, there is an epimorphism from $M$ to $L(\mu)$. We claim $M$ is a projective cover of $L(\mu)$. Otherwise, $M/R$ is a direct sum of more than two irreducible modules, where $R$ is the radical of $M$. For sufficiently large $k'$, a quotient of $P_{k'}(\mu)\cong\mf{tr}_{k'}^n(M)$ is a direct sum of more than two irreducible modules by \lemref{lem:trunc}. It contradicts to the fact that $P_{k'}(\mu)$ is a projective cover of $L_{k'}(\mu)$. Therefore $M$ is a projective cover of $L(\mu)$ and $M$ has a finite Verma flag.

Fianlly, we need to prove \eqref{BGGre-inf}. Given $\la, \mu \in \mc{P}$, we can choose a sufficient large $k$ such that $\la, \mu \in \mc{P}_k$. By the arguments above and \lemref{lem:trunc}, we have
$$(P(\la): \Delta(\mu)) = (P_k(\la): \Delta_k(\mu)).$$
By \eqref{BGGre} and \lemref{lem:trunc}, we have
$$ (P_k(\la): \Delta_k(\mu)) = [\Delta_k(\mu): L_k(\la)]= [\Delta(\mu): L(\la)].$$ Hence the proof is completed.
 \end{proof}

From the proof of \thmref{Oproj}, we have the following corollary.

\begin{cor} \label{Thm::ProjToProj}
For all $k$ and $\mu\in \mc{P}_k$, we have
$$
\mf{tr}_{k}^{\infty}(P(\mu))\cong P_k(\mu).
$$.
\end{cor}

Applying the duality functor to $\mc O$, we have the following corollary.

\begin{cor}\label{inj}
For each $\mu \in \mc{P}$, there is an injective envelope $I(\mu)$ of $L(\mu)$ in $\mc O$. Moreover, $I(\mu)$ has a finite filtration such that its successive quotients are isomorphic to costandard modules.
\end{cor}

Using the same arguments for the proof of  \cite[Proposition 5.2]{CLW11}, every module in $(\mc O^{\mf a})_f$ has a finite filtration such that its successive quotients are highest weight modules. Since every Verma module in $\mc O^{\mf a}$ has a finite composition series by \cite[Corollary 3.12]{CL10}, every  finitely generated module in $\mc O^{\mf a}$ has a finite composition series. Therefore the notions of finitely generated module and module having a finite composition series are the same in $\mc O^{\mf a}$. Hence  $(\mc O^{\mf a})_f$ is an abelian category and every projective cover of irreducible module in $\mc O^{\mf a}$ is also in $(\mc O^{\mf a})_f$.

\begin{cor} The category $(\mc O^{\mf a})_f$ has enough projectives and injectives.
\end{cor}

 \begin{proof}
Since we have the duality functor $\vee$ on $(\mc O^{\mf a})_f$, it is enough to show $(\mc O^{\mf a})_f$ has enough projectives. Since every module in $M\in (\mc O^{\mf a})_f$ has a finite composition series, there is a finite direct sum of projective covers of irreducible modules, which have a finite composition series from the disscussions above, having an epimorphic image $M$. The proof is completed.
\end{proof}

Using the existence of projective covers, we have the following proposition which is a much stronger version of \propref{stExt}. The proof is exactly the same as the proof of \cite[Proposition A3.13]{Don98}. For our application, we only need $n\in\N$. Note that every projective cover in $\mc{O}_n$ has a finite Verma flag.

\begin{prop}\label{Don} For $0\le k< n \le \infty$, $M\in \mc O^{\Delta}_{n,k}$ and $N \in \mc O_{n}$, we have isomorphisms
$$
 \mf{tr}_k^n: \emph{Ext}^i_{\mc O_{n}}(M, N) \longrightarrow \emph{Ext}^i_{\mc O_{k}}(\mf{tr}_k^n(M),\mf{tr}_k^n(N)), \quad \text{for all $i\ge 0$}.
$$
\end{prop}

\subsection{Highest weight category }\label{HWC}

In  this subsection, we show that $\mc{O}_{n,\la}$, for each $n\in\Ninf$, has a Kazhdan-Lusztig theory in the sense of \cite[Definition 2.1]{CPS92}. We also obtain some properties which will be used in next section.

For a dominant integral weight $\la\in X_n$ and $n\in \N \cup \{\infty\}$, we let
\begin{equation}\label{weights}
\Lambda^{\la,-}_n:= (W_n\cdot \la)\cap  \mcP^-_n\quad\text{and}\quad\Lambda^{\la}_n:= \Lambda^{\la,-}_n\cap \mc{P}_n.
\end{equation}
 Let $\mc{O}_{n,\la}$ (resp. $\mc{O}_{n,\la}^-$) denote the full subcategory of $\mc{O}_n$ (resp. $\mc{O}_n^-$) consisting of modules such that its irreducible subquotients are isomorphic to $L_n(\mu)$ for some $\mu\in\Lambda^{\la}_n$ (resp. $\Lambda^{\la,-}_n$).
Note that $\mc{O}_{n,\la}^-$, for $n\in \N$, is a block of the usual parabolic BGG category and $\mc{O}_{n,\la}=\mc{O}_{n,\la}^-\cap\mc{O}_n$.

Recall that $(\mcP^-_n, \le)$, for each $n\in\Ninf$, is a partially ordered set by declaring $\mu\le \la$~ ($\mu, \la\in \mcP^-_n$) if
$\la-\mu$ can be written as a nonnegative integral linear combination of simple roots of $\G_n$. Note that  every module in $\mc O_{n,\la}^-$(resp. $\mc O_{n,\la} $),  for $n \in \N$, has finite composition series.  Then it is easy to see $\mc O_{n,\la}^-$(resp. $\mc O_{n,\la} $), for each $n\in\N$, is an Artinian highest weight category in the sense of \cite[Definition 3.1]{CPS88} with repect to the poset $(\Lambda^{\la,-}_n, \le)$  (resp. $(\Lambda^{\la}_n, \le)$), simple module $L_n(\mu)$, injective envelope $I^-_n(\mu)$ (resp. $I_n(\mu)$) of $L_n(\mu)$, projective covers $P^-_n(\mu)$ (resp. $P_n(\mu)$) of $L_n(\mu)$, standard module $\Delta_n(\mu)$ and costandard module $\nabla_n(\mu)$ for  $\mu\in \Lambda^{\la,-}_n$ (resp. $\mu\in\Lambda^{\la}_n$). Recall $P^-_n(\mu)$ is defined in the proof of \propref{Oproj-finite}.
Note that ${\rm End}_{\G_n}(L_n(\mu))\cong\C$ for all $\mu\in\Lambda^{\la,-}_n$,
and we have the duality functor $\vee$ on $\mc O_{n,\la}^-$(resp. $\mc O_{n,\la} $) satisfying $\vee\circ\vee\cong {\rm id}_{\mc{O}_n}$ and $L_n(\mu)^\vee\cong L_n(\mu)$. In particular, $(\mc O_{n,\la}^-)^{\rm op}$ and $\mc O_{n,\la}^{\rm op}$ are also  Artinian highest weight categories. Therefore $\mc O_{n,\la}^-$ and $\mc O_{n,\la}$ satisfy the conditions $(1.1)$ and $(1.2)$ in \cite{CPS93}.

For $n\in\Ninf$, let $W_{n,\la} : = \{w\in W_n \, |\, w\cdot \la =\la \}$ and let $W^{\la}_n$ denote the set of the shortest representatives of left cosets in $W_n/W_{n,\la}$. Also, let
\begin{align}\label{W-ring}
\mathring{W}_n^\la : = \{w\in W^\la_n\,|\,~w\cdot \la \in \mcP^-_n\}.
\end{align} Then every weight $\mu\in \Lambda^{\la,-}_n$ can be written in the form $\mu=w\cdot \la$ for a unique $w\in \mathring{W}_n^\la$.
Define the {\em length} function $\ell\,:\,\Lambda^{\la,-}_n\longrightarrow \Z$ by letting $\ell(w\cdot\la)$ be the length of $w$, for $w\in \mathring{W}_n^\la$. It is known that the highest weight category $\mc{O}^-_n$, for $n\in \N$, has a Kazhdan-Lusztig theory with respect to the length function $\ell$ in the sense of \cite[Definition 2.1]{CPS92}. Since $\mc{O}^-_n$ has the duality functor, the fact that $\mc{O}^-_n$ has a Kazhdan-Lusztig theory is equivalent to
\begin{equation}\label{van}
 \Ext^i_{\mc{O}^-_n}(\Delta_n(\mu), L_n(\gamma))\not=0\,\,\, \Rightarrow\,\,\, \ell(\mu)-\ell(\gamma)\equiv i \,\,\, ({\rm mod }\, 2).
 \end{equation}
 There is a sketched proof of the vanishing condition for \eqnref{van} in the proof of \cite[Theroem 3.2]{CSe} (cf. \cite[Theorem 3.11.4]{BGS96}).

\begin{prop}
For $n\in \N$, $\mu\in\Lambda^{\la}_n$ and $M\in \mc{O}_n$, we have
$$ \Ext^i_{\mc{O}_n}(\Delta_n(\mu), M)\cong  \Ext^i_{\mc{O}^-_n}(\Delta_n(\mu), M),\qquad \text{for all $i\ge 0$.}$$
In particular, $\mc{O}_n$ has a Kazhdan-Lusztig theory for $n\in \Ninf$.
\end{prop}

\begin{proof} For $n\in \N$ and $\mu\in\Lambda^{\la}_n$, let $G$ (resp. $G^-$) denote the left exact functor $\Hom_{\G_n}(\Delta_n(\mu), \cdot)$ on $\mc{O}_n$ (resp. $\mc{O}_n^-$). By \propref{Oproj-finite}, there is an injective resolution
\begin{equation}\label{resol}
 0\longrightarrow M \longrightarrow I_0\longrightarrow  I_1 \longrightarrow I_2 \longrightarrow\cdots
\end{equation}
of $M\in\mc{O}_n$ in the category $\mc{O}_n$. Note that each $I_i$ has a finite filtration such that its successive quotients are costandard modules in $\mc{O}_n$. Therefore each $I_i$ is an $G^-$-acyclic module in $\mc{O}_n^-$ (see, for example, \cite[Theorem 6.12]{Hum08}) and hence \eqref{resol} is an $G^-$-acyclic resolution of $M$ in $\mc{O}_n^-$. This implies $\Ext^i_{\mc{O}_n}(\Delta_n(\mu), M)$ (resp. $\Ext^i_{\mc{O}^-_n}(\Delta_n(\mu), M)$), which is the $i$-th right derived functor of $G$ (resp. $G^-$) on $M$, can be computed as the $i$-th cohomology group of the complexes
\begin{equation*}
 0\longrightarrow \Hom_{\G_n}(\Delta_n(\mu), I_0)\longrightarrow \Hom_{\G_n}(\Delta_n(\mu), I_1) \longrightarrow \Hom_{\G_n}(\Delta_n(\mu), I_2) \longrightarrow\cdots
\end{equation*}
Therefore $ \Ext^i_{\mc{O}_n}(\Delta_n(\mu), M)\cong  \Ext^i_{\mc{O}^-_n}(\Delta_n(\mu), M)$, for all $i\ge 0$. In particular, $\mc{O}_n$ has a Kazhdan-Lusztig theory for $n\in \N$.

For $n=\infty$, let $\mu,\gamma\in \Lambda_n^\la$. Then $\mu,\gamma\in \Lambda_k^\la$ for some $k\in\N$. By \propref{Don}, we have isomorphisms
$${\rm Ext}^i_{\mc O_{n}}({\Delta}_{n}(\mu), L_n(\gamma))\cong {\rm Ext}^i_{\mc O_{k}}({\Delta}_{k}(\mu), L_k(\gamma)), \quad \text{for all $i\ge 0$}.$$ Therefore $\mc{O}_n$ also has a Kazhdan-Lusztig theory with respect to the length function $\ell$.
\end{proof}

\begin{lem}\label{dimExt-L-L}
For $0\le k< n < \infty$ and $\mu,\gamma\in\Lambda^{\la}_k$, we have
$${\rm dimExt}^j_{\mc{O}_n}(L_n(\mu), L_n(\gamma)) \ge {\rm dimExt}^j_{\mc{O}_k}(L_k(\mu), L_k(\ga)),\qquad \text{for all $j\ge 0$.}$$
\end{lem}

\begin{proof} From the discussions above, we know $\mc{O}_{n,\la}$ and $\mc{O}_{k,\la}$ are highest weight categories satisfying the assumptions of \cite[Corollary 3.9]{CPS92}. For $\mu,\gamma\in\Lambda^{\la}_k$, we have
	 \begin{align*}
	&{\rm dimExt}^j_{\mc{O}_n}(L_n(\mu), L_n(\gamma))\\
	=&\, \sum_{i+l=j} \sum_{\zeta \in\Lambda^{\la}_n}{\rm dimExt}^i_{\mc{O}_n}(\Delta_n(\zeta), L_n(\ga))     \cdot{\rm dimExt}^l_{\mc{O}_n}(\Delta_n(\zeta), L_n(\mu)) \\
\ge &\, \sum_{i+l=j} \sum_{\zeta \in\Lambda^{\la}_k}{\rm dimExt}^i_{\mc{O}_n}(\Delta_n(\zeta), L_n(\ga))     \cdot{\rm dimExt}^l_{\mc{O}_n}(\Delta_n(\zeta), L_n(\mu)) \\
		=&\, \sum_{i+l=j} \sum_{\zeta \in \Lambda^{\la}_k}{\rm dimExt}^i_{\mc{O}_k}(\Delta_k(\zeta), L_k(\ga))     \cdot{\rm dimExt}^l_{\mc{O}_k}(\Delta_k(\zeta), L_k(\mu)) \\
	=&\, {\rm dimExt}^j_{\mc{O}_k}(L_k(\mu), L_k(\ga)).
	\end{align*}
The first and last equalities follow from \cite[Corollary 3.9]{CPS92}, and the second inequality follows from \propref{Don}.
  \end{proof}

\section{Endomorphism algebra of the sum of projective covers} \label{sect::EmbedingofRings}
This section is devoted to showing that $\mc{O}$ is a Koszul category. First, we show that $\mc{O}_n$, $n\in\N$,  is a Koszul category. Second, we prove that the truncation functor $\mf{tr}^n_k$ is virtually a Schur functor. We are able to show that the endomorphism algebra $R_{k,\la}$, with the Koszul grading, of the sum of projective covers can be identified with a graded subalgebra of $R_{n,\la}$ with the Koszul grading for $k<n$. Finally, we show that  $R_{\infty,\la}$, which is isomorphic to the direct limit of Koszul algebras $R_{k,\la}$, is Koszul.

For a dominant integral weight $\la\in X_n$ and $n\in \N \cup \{\infty\}$, we let
$$P_{n,\la}^-: = \bigoplus_{\mu \in  \Lambda_{n}^{\la,-}} P^-_n(\mu)\qquad \text{and}\qquad P_{n,\la}: = \bigoplus_{\mu \in  \Lambda_{n}^\la} P_n(\mu).$$
Recall that $\Lambda^{\la,-}_n$ and $\Lambda^{\la}_n$ are defined in \eqref{weights}. For $\gamma \in \Lambda^{\la,-}_n$ (resp. $\Lambda^{\la}_n$), let $e^-_{n,\gamma} \in {\rm End}_{\G_n}(P^-_{n,\la})$ (resp. $e_{n,\gamma} \in {\rm End}_{\G_n}(P_{n,\la})$) be the natural projection onto $P^-_n(\gamma)$ (resp. $P_n(\gamma)$).
For $n\in\N$, we let
$$R^-_{n,\la}:={\rm End}_{\G_n}(P^-_{n,\la})\qquad \text{and}\qquad R_{n,\la}:={\rm End}_{\G_n}(P_{n,\la}).$$
For $n=\infty$, let
 \begin{align*}
 R_{\infty,\la}:=\bigoplus_{\mu, \gamma\in   \Lambda_{\infty}^\la}e_{\infty,\mu}{\rm End}_{\G_\infty}(P_{\infty,\la})e_{\infty,\gamma}.
  \end{align*}
  Note that $R_{n,\la}$ are finite-dimensional algebras for all $n\in\N$ while  $R_{\infty,\la}$ is an infinite-dimensional locally unital algebra.

\begin{lem}\label{R-R} For $n\in \N$ and a dominant integral weight $\la\in X_n$,  we have
$$R_{n,\la}\cong  R^-_{n,\la}/(e),$$ where $e:=\sum_{\gamma\in \Lambda_n^{\la,-}\backslash\Lambda^{\la}_n}e^-_{n,\gamma}$.
\end{lem}

\begin{proof} For $\mu\in\Lambda^{\la}_n$, let $N_\mu$ be the submodule of $P^-_n(\mu)$ defined in the proof of \propref{Oproj-finite} such  that  $P_n(\mu)= P^-_n(\mu)/N_\mu$ and let $$\pi_\mu\,:\, P^-_n(\mu)\longrightarrow P^-_n(\mu)/N_\mu=P_n(\mu)$$ be the canonical map.
For $\mu\in  \Lambda_n^{\la,-}\backslash\Lambda^{\la}_n$, let $N_\mu=P^-_n(\mu)$ and let $$\pi_\mu\,:\, P^-_n(\mu)\longrightarrow P^-_n(\mu)/N_\mu=0$$ be the zero map. Let $\pi:=\sum_{\gamma\in \Lambda_n^{\la,-}}\pi_\gamma$. For $f\in R^-_{n,\la}$, it is clear that $\pi\circ f$ sends $\bigoplus_{\mu\in \Lambda_n^{\la,-}} N_\mu$ to zero. Therefore $\pi\circ f$  factors through an unique map $\overline{f}\in {\rm End}_{\G_n}(P_{n, \la})$.  Define $\varphi\,:\, R^-_{n,\la}\longrightarrow R_{n,\la}$ by letting $\phi$ send $f$ to $\overline{f}$. By uniqueness, $\varphi$ is a homomorphism of algebras and the ideal $(e)$ is contained in the kernel of $\varphi$.

For $\mu,\gamma\in \Lambda^{\la}_n$ and $g\in \Hom_{\G_n}(P_n(\mu), P_n(\gamma))$, there is an $f\in \Hom_{\G_n}(P^-_n(\mu), P^-_n(\gamma))$ such that $g\circ \pi_\mu =\pi_\gamma \circ f$ since $\pi_\gamma$ is an epimorphism. Therefore $\varphi(f)=g$. Hence $\varphi$ is an epimorphism. For $\mu,\gamma\in\Lambda^{\la}_n$, let $f\in \Hom_{\G_n}(P^-_n(\mu), P^-_n(\gamma))$ such that $\varphi(f)=0$. Then $\pi_\gamma\circ f=0$ and hence $f(P^-_n(\mu))\subseteq N_\gamma$. From the proof of \propref{Oproj-finite}, there are projective module $\bigoplus^r_{i=s+1}P_n(\gamma_i)$ and an epimorphism $h$ from $\bigoplus^r_{i=s+1}P_n(\gamma_i)$ onto $N_\gamma$ with each $\gamma_i\in \Lambda_n^{\la,-}\backslash\Lambda^{\la}_n$. Therefore there is a homomorphism $g\in \Hom_{\G_n}(P^-_n(\mu), \bigoplus^r_{i=s+1}P_n(\gamma_i))$ such that $h\circ g=f$. Therefore $f\in (e)$ and hence  $R_{n,\la}\cong  R^-_{n,\la}/(e)$.
\end{proof}

\begin{prop}\label{R-Kosz} For $n\in \N$ and a dominant integral weight $\la\in X_n$,  $R_{n,\la}$ is a Koszul algebra.
\end{prop}

\begin{proof} First we observe that for a given module $M \in \mc{O}^-_n$, we have that $M$ belonging to $\mc{O}_n$ if and only if $\Hom_{\G_n}(P^-_n(\mu), M)=0$ for all $\mu\in \Lambda_n^{\la,-}\backslash\Lambda^{\la}_n$. Therefore, the equivalence  $\Hom_{\G_n}(P_{n,\la}^-,\cdot) : \mc{O}^-_n \longrightarrow \text{Mof-}R^-_{n,\la}$ identifies $\mc{O}_n$ with the full subcategory of $\text{Mof-}R^-_{n,\la}$ consisting of the modules killed by the ideal $(e)$ defined in \lemref{R-R}. Therefore $\mc{O}_n$ is equivalent to $\text{ Mof-}R_{n,\la}$. By \lemref{R-R} and \cite[Proposition A3.3]{Don98}, we have
$${\rm Ext}^j_{\mc{O}_n}( M,  N)={\rm Ext}^j_{\mc{O}^-_n}( M, N), \quad\text{for all $M,\, N\in \mc{O}_n$}.$$
Now the proposition follows from \cite[Proposition 3.2]{Ba99} and \cite[Lemma 2.2]{SVV}.
\end{proof}

\begin{cor}\label{En-Kos}
For $n\in \N$, the extension algebra
$${\rm Ext}^\bullet_{\mc{O}_n}(\bigoplus_{\mu\in\Lambda^{\la}_n} L_n(\mu), \bigoplus_{\mu\in\Lambda^{\la}_n} L_n(\mu))$$
is a Koszul algebra.
\end{cor}

\begin{lem}\label{Ext1mono}
For $0\le k< n \le \infty$ and $\mu,\gamma\in \mc{P}_k$, the truncation
 \begin{align*}\mf{tr}^n_k: {\rm Ext}^1_{\mc{O}_n}(L_n(\mu), L_n(\gamma))\longrightarrow {\rm Ext}^1_{\mc{O}_k}(L_k(\mu), L_k(\gamma)).
  \end{align*}
is a monomorphism.
\end{lem}
\begin{proof}
Let
$$0\longrightarrow L_n(\gamma)\longrightarrow M \longrightarrow L_n(\mu)\longrightarrow 0$$
be an exact sequence corresponding to a nonzero element in ${\rm Ext}^1_{\mc{O}_n}(L_n(\mu), L_n(\gamma))$. Since the truncation functor $\mf{tr}^n_k$ commutes with the duality functor $\vee$ and the dual of a non-split short exact sequence is also non-split, we may assume that $\mu\nless\gamma$.  Then $M$ is a highest weight module of highest weight $\mu$ and hence the truncation of the short exact sequence is also non-split.
\end{proof}

For $n\in\N$, let
$$E_{n,\la}:={\rm Ext}^\bullet_{\mc{O}_n}(\bigoplus_{\mu\in\Lambda^{\la}_n} L_n(\mu),\bigoplus_{\mu\in\Lambda^{\la}_n} L_n(\mu))$$

\begin{lem}\label{ExtLonto}
For $0\le k< n < \infty$, we have a graded epimorphism of graded algebras
 \begin{align*} &\eta^n_k: {\rm Ext}^\bullet_{\mc{O}_n}(\bigoplus_{\mu\in\Lambda^{\la}_k} L_n(\mu), \bigoplus_{\mu\in\Lambda^{\la}_k} L_n(\mu))\longrightarrow E_{k,\la}
  \end{align*}
 induced from $\mf{tr}^n_k$. In particular, we have
$$\mf{tr}^n_k: {\rm Ext}^j_{\mc{O}_n}( L_n(\mu),  L_n(\gamma))\longrightarrow {\rm Ext}^j_{\mc{O}_k}( L_k(\mu), L_k(\ga)) $$
is an epimorphism for all $\mu, \gamma\in\Lambda^{\la}_k$ and $j\ge 0$.
\end{lem}

\begin{proof}
	Let $E_{k}^j={\rm Ext}^j_{\mc{O}_k}(\bigoplus_{\mu\in\Lambda^{\la}_k} L_k(\mu), \bigoplus_{\mu\in\Lambda^{\la}_k} L_k(\mu))$ for $j\ge 0$. By \propref{R-Kosz} and \cite[Theorem 2.10.2]{BGS96}, the extension algebra $\bigoplus_{j\ge 0}E_{k}^j$ is a Koszul algebra and hence   generated by the subspace $E_{k}^1$ over the subalgebra $E_{k}^0$. But
$$\mf{tr}^n_k: {\rm Ext}^1_{\mc{O}_n}( L_n(\mu),  L_n(\gamma))\longrightarrow {\rm Ext}^1_{\mc{O}_k}( L_k(\mu), L_k(\ga)) $$
is an isomorphism for all $\mu, \gamma\in\Lambda^{\la}_k$ by \lemref{dimExt-L-L} and \lemref{Ext1mono}. Certainly, the image of $\eta^n_k$ contains $E_{k}^0$ and hence  $\eta^n_k$ is an epimorphism.
  \end{proof}

For $0\le k< n \le \infty$, let $$e_{n,k}:=\sum_{\mu, \gamma\in  \Lambda_{k}^\la}e_{n,\mu} \qquad \text{and}\qquad R_{n,k,\la}:=e_{n,k}R_{n,\la}e_{n,k}.$$
As a direct consequence of \propref{Oproj-finite}, \corref{Thm::ProjToProj} and \thmref{Thm::Equivlanece}, we have the following proposition. Note that \eqref{BGGre-inf} and \eqref{BGGre} implies
 $(P_n(\mu): \Delta_n(\gamma))=0$ for $\mu \in \Lambda^\la_k$, $\gamma \in \Lambda^\la_n\backslash\Lambda^\la_k$ and  $k< n$.

\begin{prop}\label{EndP=}
For $0\le k<n\le \infty$ and a dominant integral weight $\la \in X_n$, we have an isomorphism of algebras
 \begin{align*}\psi_{k}^{n}\, :\, R_{n,k,\la}\longrightarrow R_{k,\la}
 \end{align*}
 induced from the $\mf{tr}^n_k$.
\end{prop}

For $n\in\N$ and a dominant integral weight $\la \in X_n$,   let $$F_n\,:\,  \mc O_{n,\la} \longrightarrow \text{Mof-}R_{n,\la}$$ be the equivalence defined by $F_n(\cdot) : = \text{Hom}_{ \mf g_n}(P_{n,\la}, \cdot)$. Note that $\mc O_{n,\la}$ has only finitely many non-isomorphic irreducible modules.
Let  $${S}^n_k: \text{Mof-}R_{n,\la}\longrightarrow \text{Mof-}R_{n,k,\la}$$ denote the {\em Schur} functor by defining ${S}^n_k(M):= Me_{n,k}$, for $M \in \text{Mof-}R_{n,\la}$ and ${S}^n_k(f)\in \Hom_{R_{n,k,\la}}(Me_{n,k}, Ne_{n,k})$, to be the restriction of $f$ to $Me_{n,k}$ for $f\in \Hom_{R_{n,\la}}(M, N)$. Recall that $\psi_k^{n}$ is an isomorphism defined in \propref{EndP=}. Let $$\psi_k^{n*}\,:\,\text{Mof-}R_{k,\la}\longrightarrow \text{Mof-}R_{n,k,\la}$$ denote the functor induced from the isomorphism $\psi_k^{n}$ defined by $\psi_k^{n*}(M)=M$ for $M \in \text{Mof-}R_{k,\la}$ with the action of $R_{n,k,\la}$ given by letting $$va:=v\psi_k^{n}(a),$$ for $v\in M$ and $a\in R_{n,k,\la}$,  and $\psi_k^{n*}(f)=f$ for $f\in \Hom_{R_{k,\la}}(M, N)$.

\begin{prop} \label{TrSchur}
 There is a natural isomorphism from the functor $S^n_k\circ F_n$ to  the functor $\psi_k^{n*}\circ F_k \circ \mf{tr}^n_k$.
\end{prop}
\begin{proof}
For a given $M \in \mc O_{n,\la}$, we have $${S}^n_k(F_n(M)) = \Hom_{\mf g_n}(P_{n,\la},M)e_{n,k},$$ and $$\psi_k^{n*}\circ F_k \circ \mf{tr}^n_k(M)= \Hom_{\mf g_k}(P_{k,\la}, \mf{tr}^n_k (M)).$$
By \corref{Thm::ProjToProj} and \propref{Don}, we have an isomorphism
  $$\theta_M\,:\, \Hom_{\mf g_n}(P_{n,\la},M)e_{n,k}\longrightarrow\Hom_{\mf g_k}(P_{k,\la}, \mf{tr}^n_k (M))$$
  of vector spaces induced from $\mf{tr}^n_k$. Since $\mf{tr}^n_k$ is a functor, the linear isomorphism $\theta_M$ becomes an isomorphism of right $R_{n,k,\la}$-modules. Now it is also clear that $$(\psi_k^{n*}\circ F_k \circ \mf{tr}^n_k)(f)\circ\theta_M=\theta_N\circ (S^n_k\circ F_n)(f)$$ for all $f\in \Hom_{\mc{O}_n}(M,N)$ and $M, N \in \mc O_{n,\la}$.
\end{proof}

For $n\in \N$ and a dominant integral weight $\la\in X_n$, $R_{n,\la}$ is a Koszul algebra by \propref{R-Kosz}.  We fix the grading on $R_{n,\la}$ to be the Koszul grading and regard $R_{n,k,\la}$ as the graded subalgebra of $R_{n,\la}$ for each $k<n$. Note that each element $e_{n,k}$ is an idempotent element of degree zero.

\begin{prop}\label{gr-R=eRe}
For $0\le k<n< \infty$ and a dominant integral weight $\la \in X_n$, the isomorphism of algebras
 \begin{align*}\psi_{k}^{n}\, :\, R_{n,k,\la}\longrightarrow R_{k,\la}
 \end{align*}
obtained in \propref{EndP=} is a graded isomorphism.
\end{prop}

\begin{proof} Note that $R_{n,k,\la}$ and $R_{n,k,\la}^{\rm op}$ are  Koszul algebras by \propref{R-Kosz} and \propref{EndP=} and \cite[Proposition 2.2.1]{BGS96}. By the uniqueness of Koszul grading \cite[Corollary 2.5.2]{BGS96}, it is enough to show the Koszul grading on $R_{n,k,\la}$ equals to the grading induced from $R_{n,\la}$. Let $R=R_{n,\la}$ and $A=R_{n,k,\la}$. For $\mu\in \mc{P}_k$, let $\mc{L}(\mu):=\Hom_{\G_n}(P_{n,\la}, L_n(\mu))$ be the $1$-dimensional right $R$-module concentrated in degree $0$ and let $\mc{L}(\mu)$ also denote the $1$-dimensional graded right $A$-module $S^n_k(\mc{L}(\mu))$ with induced grading from $\mc{L}(\mu)$.

For $i\ge 0$ and $\mu,\gamma\in \mc{P}_k$, there are natural maps $$j\,:\,\Ext^i_{\text{-}R} (\mc{L}(\mu),\mc{L}(\gamma))\longrightarrow \Ext^i_{\text{-}A} (\mc{L}(\mu),\mc{L}(\gamma)\langle i\rangle)$$
and
$$\overline{j}\,:\,\ext^i_{\text{-}R} (\mc{L}(\mu),\mc{L}(\gamma)\langle i\rangle)\longrightarrow \ext^i_{\text{-}A} (\mc{L}(\mu),\mc{L}(\gamma)\langle i\rangle) $$
induced from the Schur functor $S^n_r$. By \lemref{ExtLonto} and \propref{TrSchur}, $j$ is an epimorphism.  We claim that $\overline{j}$ is also an epimorphism. By the proof of \cite[Proposition 3.9.2]{BGS96}, we have
\begin{equation}\label{E-e}
\Ext^i_{\text{-}R}(M,N) =\bigoplus_{s\in \Z}\ext^i_{\text{-}R} (M,N\langle s\rangle)\quad \text{for $i\ge 0$ and $M, N\in \text{mof-}R$.}
\end{equation}
Since $R^{\rm op}$ is a Koszul algebra, we have $\Ext^i_{\text{-}R}(\mc{L}(\mu),\mc{L}(\gamma)) =\ext^i_{\text{-}R} (\mc{L}(\mu),\mc{L}(\gamma)\langle i\rangle)$ for $i\ge 0$. Since $j$ is an epimorphism, we have  $${\rm dim}\, \ext^i_{\text{-}A} (\mc{L}(\mu),\mc{L}(\gamma)\langle i\rangle)\ge {\rm dim}\,\Ext^i_{\text{-}A}(\mc{L}(\mu),\mc{L}(\gamma))\quad \text{for all $i\ge 0$.}$$ Applying \eqref{E-e} to our situation, we have
$\ext^i_{\text{-}A} (\mc{L}(\mu),\mc{L}(\gamma)\langle j\rangle)=0$ for $j\not=i$. By \cite[Propositions 2.1.3, and 2.2.1]{BGS96}, $A$ with the grading induced from $R$ is a Koszul algebra. The proof is completed.
\end{proof}

The Koszulity of $R_{\infty,\la}$ will follow from the following lemma.

\begin{lem} \label{UKoszul}
Let $A$ be a locally unital algebra and let $I$ be the subset of $A$ consisting of the mutually orthogonal idempotents such that $A=\bigoplus_{e_\mu,e_\beta\in I}e_\mu A e_\beta$ and $A_0=\bigoplus_{e_\mu\in I}\C e_\mu$. Assume that there is an
increasing sequence $\{I_i\}_{i\ge 1}$ of subsets of $I$ such that $\bigcup_{i\ge 1} I_i=I$. If $B^i:=\bigoplus_{e_\mu,e_\beta\in I_i}e_\mu A e_\beta$ are locally unital Koszul algebras with ${\rm span} \{e_\mu\,|\, e_{\mu}\in I_i\}= (B^i)_0$ for all $i\ge 1$ such that their gradings are compatible, then $A$ is a locally unital Koszul algebra such that $B^i$ are graded subalgebras of $A$.
\end{lem}
\begin{proof}
It is clear that $\bigcup_{i\ge 1} B^i=A$ and hence $A$ is a positively graded algebra by defining $A_j:=\bigcup_{i\ge 1} B^i_j$ for all $j\ge 0$ and $A_0={\rm span} \{e_\mu\,|\, \mu\in I\}$.
For every $n\in \N$, let the following exact sequence of  graded $B^n$-modules
\begin{align*}
 C\langle n\rangle:\qquad \cdots\longrightarrow P^{n, j}\stackrel{f^{n,j}}\longrightarrow P^{n, j-1}\cdots \longrightarrow P^{n, 2}\stackrel{f^{n,2}}\longrightarrow P^{n, 1} \stackrel{f^{n,1}}\longrightarrow P^{n, 0}\stackrel{f^{n,0}}\longrightarrow B^{n}_0 \longrightarrow 0
\end{align*}
be a graded projective resolution of $B^{n}_0$ in $B^{n}$-mod
  such that $P^{n, i}$ are generated by its $i$-th component for all $i\in \Z_+$. We may assume that $ P^{n, 0}=B^n$, $P^{n, i}$ and $f^{n, i}$ are defined inductively on $i$ by setting $P^{n, i}:=B^n\otimes_{B^n_0}{\rm Ker}(f^{n,i-1})_{i}$ and $f^{n, i}(b\otimes v):= bv$ for all $b\in B^n$ and $v\in {\rm Ker}(f^{n,i-1})_{i}$.

  First we observe that $ C\langle n-1\rangle$ can be regarded as a graded subcomplex of $ C\langle n\rangle$ of graded $B^{n-1}$-modules by induction. We consider the following diagram:
    \begin{eqnarray*}
\CD
\cdots P^{n-1, j}  @>f^{n-1,j}>> P^{n-1, j-1} @>>> \cdots P^{n-1,1} @>f^{n-1,1}>> P^{n-1,0} @>f^{n-1,0}>> B^{n-1}_0 @>>>0\\
 @VV \tau_j V @VV\tau_{j-1}V  @VV\tau_1 V @VV\tau_0 V @VV\tau V @.\\
\cdots P^{n, j}  @>f^{n,j}>> P^{n, j-1} @>>> \cdots P^{n,1} @>f^{n,1}>> P^{n,0} @>f^{n,0}>> B^{n}_0 @>>>0,
 \endCD
\end{eqnarray*}
where $\tau$ and $\tau_0\,:\, B^{n-1}\longrightarrow B^n$ are the natural inclusion maps, and the $B^{n-1}$-homomorphism $\tau_{k} : P^{n-1, k}\longrightarrow P^{n,k}$ are defined inductively on $k$ by $$\tau_{k}(b\otimes v):=\tau_0(b)\otimes\tau_{k-1}(v)=b\otimes\tau_{k-1}(v)$$ for all $b\in B^{n-1}, v\in {\rm Ker}(f^{n-1,k-1})_{k}$, for each $k\ge 1$.  To show that the  diagram commutes, we need to show the following diagram
 \begin{eqnarray*}
\CD
P^{n-1, j+1}  @>f^{n-1,j+1}>> P^{n-1, j} \\
 @VV \tau_{j+1} V @VV\tau_{j}V  \\
P^{n, j+1}  @>f^{n,j+1}>> P^{n, j}
 \endCD
\end{eqnarray*}
commutes for every $j\ge 0$. Note that
$$
f^{n,j+1}(\tau_{j+1}(b\otimes v))=b\tau_{j}(v)=\tau_{j}(bv)=\tau_{j}(f^{n-1,j+1}(b\otimes v)),
$$
for all $b\in B^{n-1}, v\in {\rm Ker}(f^{n-1,j+1})_{j+2}$, for every $j\ge 0$. Therefore the diagram commutes. Finally, we need to show $\tau_{j}$ are injective. Assume that  $\tau_{j}$ are injective. We have a $B_0^{n-1}$-monomorphism $$\tau_{j}\,:\,{\rm Ker}(f^{n-1,j})_{j+1}\longrightarrow {\rm Ker}(f^{n,j})_{j+1}.$$ Since $B^n$ decomposes into direct sum of two $(B^{n-1}, B^n_0)$-modules $\bigoplus_{\mu\in I_{n-1}}B^{n}e_\mu$ and $\bigoplus_{\mu\in I_n\backslash I_{n-1}}B^{n} e_\mu$,  $\tau_{j+1}$ is a monomorphism. Therefore the vertical arrows are monomorphisms and hence $ C\langle n-1\rangle$ can be regarded as a graded subcomplex of $ C\langle n\rangle$ for all $n\ge 1$.

 Consider the following complex of  graded $A$-module
\begin{align*}
 C\langle \infty\rangle:\qquad \cdots\longrightarrow P^{ j}\stackrel{f^{j}}\longrightarrow P^{ j-1}\cdots \longrightarrow P^{ 2}\stackrel{f^{2}}\longrightarrow P^{ 1} \stackrel{f^{1}}\longrightarrow P^{0}\stackrel{f^{0}}\longrightarrow A_0 \longrightarrow 0
\end{align*}
where $ P^{ 0}:=A$, $f^0$ is the natural projection, $P^{ i}$ and $f^{i}$ are defined inductively by $P^{ i}:=A\otimes_{A_0}{\rm Ker}(f^{i-1})_{i}$ and $f^{i}(a\otimes v)=av$, for all $a\in A$, $v\in{\rm Ker}(f^{i-1})_{i}$ and for all $i\ge 1$. From the observation above, we see that $P^{ i}$ can be identified with $\bigcup_{n\ge 1}P^{n,i}$ and the restriction of $f^i$ to $P^{n,i}$ equals $f^{n,i}$ for all $n\ge 1$. Since  ${\rm Ker}(f^{n,i})$ is generated by elements of  degree $i$ over $B^n$ for all $i\ge 0$ and $n\ge 1$,  ${\rm Ker}(f^{i})$ are generated by degree $i$ elements over $A$ for all $i\ge 0$. Therefore the complex $ C\langle \infty\rangle$ is exact and hence $A$ is a Koszul algebra.
\end{proof}

We are now in a position to prove the main theorem.
\begin{thm} \label{Thm::RIsKoszul}
For a dominant integral weight $\la \in X$,  $R_{\infty,\la}$ is a Koszul algebra. Moreover, there is a graded isomorphism from $e_{\infty, k}R_{\infty,\la}e_{\infty, k}$ to $R_{k,\la}$ induced from $\mf{tr}^\infty_k$ for each $k\ge 1$.
\end{thm}
\begin{proof} Let $e_k:=e_{\infty, k}$ for all $k\ge 1$ and $R=R_{\infty,\la}$. For $k\ge 1$, let the grading of $e_k R e_k$ be the grading obtained from the Koszul grading of $R_{k,\la}$ through the isomorphism defined in \propref{EndP=}. By \propref{gr-R=eRe}, the grading of $e_k R e_k$ and the grading of $e_n R e_n$ are compatible for all $k\le n$. Since $R=\bigcup_{k\ge 1}e_k R e_k$, we can define the grading of each element $v$ in $R$ to be the grading of $v$ in $e_k R e_k$ for $k\gg 0$. Now the theorem is a direct consequence of \lemref{UKoszul}.
\end{proof}

\subsection{Category $\overline{\mc{O}}$ and $\widetilde{\mc{O}}$}\label{subsection::SO}
The abelian categories $\overline{\mc{O}}(d)$ and $\widetilde{\mc{O}}(d)$  consisting of certain modules over Lie superalgebras of infinite rank introduced in \cite{CL10} and \cite{CLW11} are equivalent to the category $\mc{O}(d)$ consisting of certain modules over Lie algebra of infinite rank. Our category $\mc{O}(d)$ is indeed a subcategory of category $\mc{O}(d)$ defined in there (cf. \cite[Section 6.2.3]{CW08}) consisting of integral weights modules. We can also define similarly the subcategories $\overline{\mc{O}}_{int}(d)$ and $\widetilde{\mc{O}}_{int}(d)$ of $\overline{\mc{O}}(d)$ and $\widetilde{\mc{O}}(d)$ defined in \cite[Section 6.2.3]{CW08} consisting of integral weights modules, respectively. The arguments in there also show that our categories $\mc{O}(d)$, $\overline{\mc{O}}_{int}(d)$ and $\widetilde{\mc{O}}_{int}(d)$ are equivalent.

\begin{thm} \label{Thm::InftySuperBlockKoszul} For $d\in\hf\Z$, each irreducible module in $\mc{O}(d)$ (resp. $\overline{\mc{O}}_{int}(d)$ and $\widetilde{\mc{O}}_{int}(d)$) has projective cover and injective envelope which have finite filtrations such that their successive quotients are isomorphic to Verma modules and costandard modules, respectively. Moreover, $\mc{O}(d)$
$\overline{\mc{O}}_{int}(d)$ and $\widetilde{\mc{O}}_{int}(d)$ are Koszul categories. \end{thm}

\begin{proof}
By (ii) of the definition of $\mc{O}$,  $\mc{O}(d)$ is equivalent to the category
$\bigoplus_{\la\in J(d)}\mc{O}_{\la}$, where $\mc{O}_{\la}:=\mc{O}_{\infty,\la}$ and $J(d)$ is a set  consisting of all dominant integral weights $\la\in X$ such that $\la(K)=d$. Therefore, the theorem holds for the category  $\mc{O}$ by \thmref{Oproj}, \corref{inj} and \thmref{Thm::RIsKoszul}. Since $\mc{O}(d)$, $\overline{\mc{O}}_{int}(d)$ and $\widetilde{\mc{O}}_{int}(d)$ are equivalent, the theorem also holds for the categories  $\overline{\mc{O}}_{int}(d)$ and $\widetilde{\mc{O}}_{int}(d)$.
\end{proof}

\section{The dual category} \label{Set::final}
In this section, we shall study the category $\mc {O}'_{n,\la}$, $n\in\N$. Its endomorphism algebra of the direct sum of projective covers of irreducible modules is isomorphic to the Koszul dual of the Koszul algebra $R_{n,\la}$, defined in the last section, for $n\gg 0$. We show the limit category of $\mc {O}'_{n,\la}$ for $n\rightarrow\infty$ is a Koszul category.
Since results and their proofs in the section are similar to the proof in earlier sections, we only collect the analogous  results without proof. We give an idea on how to prove the results  about the truncation functor in the section for $n\in\N$ in the proof of \propref{lem:trunc'}.

\begin{lem} \label{O=O}
Let $\la$ be a dominant integral weight in $X_\infty$. There is a positive integer $k_\la$ such that for every integer $n\ge k_\la$ we have
\begin{equation*}
       \mc{O}_{n,\la}=\mc{O}_{n,\la}^-, \qquad \text{for $n\ge k_\la$}.
\end{equation*}
\end{lem}

\begin{proof} Recall that $\delta_j$ are defined in \eqnref{delta}. It is clear that there is a positive number $k$ such that $\{(\la+\rho_n)(E_n-\delta_n K)\}_{n\ge k}$ is a strictly decreasing sequence of negative numbers. Therefore we can choose $k_\la>k$ such that the following equalities hold, for any given fixed integer $n\ge k_\la$,
\begin{equation*}
        (\la+\rho_n)(E_n-\delta_n K)\le \pm(\la+\rho_n)(E_j-\delta_j K), \qquad \text{for all $j\in \I^+_m(n)$}.
\end{equation*}
Note that for each $i\in\I^+_m(n) $ and $\mu\in X_n$, we have
$$
w(\mu)(E_i-\delta_i K)=\pm\mu(E_j-\delta_j K), \qquad \text{for some $j\in \I^+_m(n)$.}
$$
Therefore
\begin{eqnarray*}
   w\cdot\la(E_n)&=&w(\la+\rho_n)(E_n-\delta_n K) +w(\la+\rho_n)(\delta_n K)-\rho_n(E_n) \\
  &=&\pm(\la+\rho_n)(E_j-\delta_j K)+\la( K)-\rho_n(E_n) \qquad \text{(for some $j\in \I^+_m(n)$)} \\
   &\ge&(\la+\rho_n)(E_n-\delta_n K)+\la( K)-\rho_n(E_n) \qquad \text{(by assumption)}\\
   &=&\la(E_n)\ge 0.
\end{eqnarray*}
It implies that $\Lalanne\subseteq \Lalan$. The proof is completed.
\end{proof}

We choose a fixed $\phi=\sum_{i=-m}^{\infty}\phi_i\epsilon_i \in \h^*$ such that $\la(\alpha^{\vee})\in\Z$ and
$$Y_\infty=\{\alpha \in \Pi(\G_\infty)\, |\, (\phi+\rho)( \alpha^\vee) =0\},$$
where $Y_\infty$ is defined in \secref{sec:para}. Then we have
$$Y_n=\{\alpha \in \Pi(\G_n)\, |\, (\phi+\rho_n)( \alpha^\vee) =0\}\qquad \text{ for $n\in \N$}.$$  We can further assume that there is a negative half integer $a$ such that $ (\phi+\rho)(E_j)=a$, for $j\ge 1$ and
\begin{equation*}
       a\le \pm(\la+\rho)(E_j), \qquad \text{for all $j\in \I^+_m(n)$}.
\end{equation*}

\vskip 0.3cm
\noindent{\bf In the remainder of this section, we shall fix a $\phi$ as above and a dominant integral weight $\la\in X_\infty$.}

\vskip 0.3cm

For  $n\in \Ninf$, we let
$Y'_n:=\{\alpha \in \Pi_n \,|\, (\la+\rho_n)(\alpha^\vee) =0\}$, $\mf l'_n$  the Levi subalgebra associated to the set $Y'_n$ and $\mf q'_n:=\mf{l}'_n +\mf{b}_n$ the corresponding parabolic subalgebra with
nilradical $\mf{u}'_n$.
For $\mu\in \h^*_n$ satisfying $\mu(\alpha^\vee)\in \Z_+$ for all $\alpha\in Y'_n$, let $L(\mf{l}'_n,\mu)$ denote the highest weight
irreducible $\mf{l}'_n$-module of highest weight $\mu$. We extend
$L(\mf{l}'_n,\mu)$ to a $\mf{q}_n'$-module by letting $\mf{u}'_n$ act
trivially.  Define as usual the parabolic Verma module
$\Delta'_n(\mu)$ and its irreducible quotient $L'_n(\mu)$ over $\G_n$ by
\begin{align*}
\Delta'_n(\mu):=\text{Ind}_{\mf{q}'_n}^{\G_n}L(\mf{l}'_n,\mu) \qquad \text{and} \qquad
\Delta'_n(\mu) \twoheadrightarrow L'_n(\mu).
\end{align*}

Let $k'_\la$ be the minimal nonnegative integer such that $\la(E_{k'_\la})=0$.
Let $n_0$ be a fixed number satisfying
$$n_0\ge k'_\la\quad \text{and}\quad n_0\ge k_\la,$$
where $k_\la$ is defined in \lemref{O=O}.
For $n_0\le n\le \infty$, let
$$\Lpn=\{\mu\in  \h^*_n\,|\,\mu=w\cdot \phi,\, \exists w\in W_n \,\,\,\text {and}\,\,\, \mu(\alpha^\vee)\in \Z_+\,\,\, \text{for all $\alpha\in Y'_n$}\}$$ and let $\Opn:={\mc{O}^{\la}_{n,\phi}}$ denote the full subcategory of the category of $\G_n$-modules  such that every module $M$ in $\mc O'_{n}$ is a semisimple
${\h}_n$-module with finite-dimensional weight spaces $M_\gamma$, $\gamma\in \h^*_n$, satisfying the following conditions:
   \begin{itemize}
\item[(i)]  ${M}$  over ${\mf{l}'_n}$ decomposes into a direct sum of
$L({\mf{l}'_n},\mu)$ for some $\mu\in\h^*_n$ satisfying $\mu(\alpha^\vee)\in \Z_+,\,\,\, \text{for all $\alpha\in Y'_n$}$;
\item[(ii)] every irreducible subquotient is of the form $L'(\mu)$ for some $\mu\in \Lpn$.
  \end{itemize}
  Note that $\Opn$ is a block of a parabolic BGG category such that $\Delta'_n(\mu)$, $L'_n(\mu)\in \Opn$, for all $\mu\in \Lpn$. For $n\in \N$,  every module $M\in \Opn$ has a composition series, and  $\Opn$ has enough  projectives and injectives. For $1\le k <n\le \infty $, as before, $\Lpk$ is regarded as a subset of $\Lpn$ by setting $\mu(E_j)=a+j$ for all $j> k$.

For $n_0\le k <n\le \infty$, we define a functor $\trp\,:\, \Opn\longrightarrow \Opk$ by
$$\trp(M):=\bigoplus_{\mu\in \mf h^*_n,\, \mu(E_r)=a+r, \forall r> k} M_\mu$$ sending each homomorphism to its restriction. The exact functor $\trp$ sending each module to a module in $\Opk$ will be clarified by \propref{lem:trunc'} below.

\begin{prop}\label{lem:trunc'}
For $n_0\le k<n\le \infty$, $\mu \in  \Lpn$ and $Z=L'$ or $\Delta'$, we have
\begin{equation*}
\begin{aligned}\mf{tr'}_k^n\big{(}Z_n(\mu)\big{)} =
\begin{cases}
Z_k (\mu),&\quad\text{if }
\mu \in  \Lpk ;\\
0,&\quad\text{otherwise}
\end{cases} \end{aligned}
\end{equation*}
and
\begin{equation*}
\mf{tr'}_k^n\big{(}P'_n(\mu)\big{)} =P'_k(\mu),\qquad\text{if }
\mu \in  \Lpk.
\end{equation*}
\end{prop}

\begin{proof}
First, we assume $n\in\N$. By our choice of $\phi$ and the similar argruments in the proof of \lemref{O=O}, we have that $w\cdot\phi(E_n)\ge \phi(E_n)$ for all $w\in W_n$. Therefore $w\cdot\phi (E_n)\ge \phi (E_n)$ for all $w\in W_n$ and hence $M_\mu\not=0$ implies $\mu(E_n)\ge a+n$ for all $M\in \mc O'_{n,\phi}$ and $\mu\in \h^*_n$.
Therefore, the functor ${\mf{tr}'}^n_{n-1}$ has the same situation as the functor $\mf{tr}^n_{n-1}$ defined in \eqref{tr}. So the conclusion of the proposition holds for ${\mf{tr}'}^n_{n-1}$ and hence it holds for
${\mf{tr}'_k}^n:={\mf{tr}'}^n_{n-1}\circ{\mf{tr}'}^{n-1}_{n-2}\cdots \circ{\mf{tr}'}^{k+1}_k$ for $n_0\le k<n<\infty$.  For $n=\infty$, the proof is similar to \cite[Corollary 6.8]{CW12}.
\end{proof}

For $n_0\le k<n\le \infty$, let $\mc{O}^{\Delta'}_{n,\phi}$ denote
 the full subcategory of $\mc{O}'_{n,\phi}$ consisting of all $\mf g_n$-modules equipped with finite Verma flags and let $\mc{O}^{\Delta'}_{n,k,\phi}$ denote the full subcategory of  $\mc O^{\Delta'}_{n,\phi}$ of all $\mf g_n$-modules $M$ with finite filtration of $\mf g_n$-modules
\begin{align*}
0={M}_{r}\subset {M}_{r -1} \subset\cdots \subset M_0 =M,
\end{align*} satisfying
$M_{i}/M_{i+1} \cong \Delta'_n(\gamma^i)$ and $\gamma\in \Lpk$, for $0\leq i\leq r-1$.
The following proposition is an analog of \thmref{Thm::Equivlanece}.
\begin{prop} For $n_0\le k<n\le \infty$, the truncation functor
 $$ \mf{tr'}_{k}^n :  \mc O^{\Delta'}_{n, k,\phi}\longrightarrow  \mc O^{\Delta'}_{k,\phi}$$ is an equivalence.
\end{prop}

For $n\in \Ninf$ and $\ga\in  \Lpn$, let $P'_n(\ga)$ be the projective cover of $L'_n(\gamma)$ in $\Opn$ if it exists. We know $P'_n(\ga)$ exists for $n\in\N$.
The following proposition is the analogue of \thmref{Oproj}, \propref{Oproj-finite} and \corref{Thm::ProjToProj}.

\begin{prop} For $n\in \Ninf$ and $\mu \in  \Lpn$, there is a projective cover $P'_n(\mu)$ of $L_n(\mu)$ in $\mc O'_n$. Moreover, we have
$$
\mf{tr'}_{k}^{n}(P'_{n}(\mu))\cong P'_k(\mu), \quad \text {for $n_0\le k< n \le \infty$ and $\mu\in \Lpk$}.
$$
\end{prop}

For $n > n_0$ as before, we set $$P'_{n,\phi}: = \bigoplus_{\mu \in  \Lpn} P'_n(\mu)\qquad \text{and}\qquad L'_{n,\phi}: = \bigoplus_{\mu \in  \Lpn} L'_n(\mu).$$
For $n\in \Ninf$ and  $\gamma \in \Lpn$, let $e'_{n,\gamma} \in {\rm End}_{\G_n}(P'_{n, \phi})$ be the natural projection onto $P'_n(\gamma)$, and let $f'_{n,\gamma}$ be the identity map of $L'_n(\gamma)$.
For $n\in \N$, let
$$R'_{n,\phi}:={\rm End}_{\G_n}(P'_{n,\phi})\quad\text{and}\quad  E'_{n,\phi}:={\rm Ext}^\bullet_{\mc{O}'_n}(L'_{n,\phi}, L'_{n,\phi}).$$
We know $R'_{n,\phi}$ and $E'_{n,\phi}$ are Koszul algebras by \cite[Proposition 3.2]{Ba99} and \cite[Theorem 2.10.2]{BGS96}.
For $n=\infty$, let
 \begin{align*}
 R'_{\infty,\phi}:=\bigoplus_{\mu, \gamma\in   \Lambda'_{\infty,\phi}}e'_{\infty,\mu}{\rm End}_{\G_\infty}(P'_{\infty,\phi})e'_{\infty,\gamma}.
  \end{align*}

The following two propositions are the analogs of the \propref{EndP=}, \propref{gr-R=eRe} and \lemref{ExtLonto}.

\begin{prop}\label{gr-R=eRe'}
For $n_0\le k<n\le \infty$, we have an isomorphism of graded algebras
 \begin{align*}{\psi'_{k}}^{n}: \bigoplus_{\mu, \gamma\in  \Lpk}e'_{\mu}R'_{n,\phi}e'_{\gamma}
\longrightarrow R'_{k,\phi}
  \end{align*}
 induced from the truncation functor $\trp$.
 Moreover, ${\psi'_{k}}^{n}$ is a graded isomorphism for $n\in \N$.
\end{prop}

\begin{prop}\label{ExtL'onto}
For $n_0\le k<n< \infty$, we have an epimorphism of graded algebras
 \begin{align*}{\eta'_k}^{n}: \bigoplus_{\mu, \gamma\in  \Lpk}f'_{\mu}E'_{n,\phi}f'_{\gamma}\longrightarrow E'_{k,\phi}
  \end{align*}
 induced from the truncation functor $\trp$.
\end{prop}

The following is an analog of \thmref{Thm::RIsKoszul}.

\begin{prop}
 $R'_{\infty,\phi}$ is a Koszul algebra. Moreover, there is a graded isomorphism from $e'_{\infty, k}R'_{\infty,\phi}e'_{\infty, k}$ to $R'_{k,\phi}$ induced from $\mf{tr'}^\infty_k$ for each $k\ge n_0$.
\end{prop}

\begin{rem}
We assume that $\G^\xx = \G^\mf{b}$, $\G^\mf{c}$ or $\G^\mf{d}$.
For $n\in\N$, let $w_{n,0}$ denote the longest element in $W_n$.
Note that  $-w_{n,0}$ is the identity in $W_n$ for $\G_n=\G^{\mf{b}}_n,\, \G^{\mf{c}}_n$, or $\G_n=\G^{\mf{d}}_n$ such that $m+n$ is an even number (see, for example, \cite[Exercise 4.10]{BB}).
Applying \cite[Theorem 3.7]{Ba99} to our situations, there are graded isomorphisms between Koszul algebras
\begin{equation}\label{R-E}
R_{n,\la} \cong E'_{n,\phi}\quad\text{and}\quad
E_{n,\la}\cong R'_{n,\phi}
\end{equation}
for $n\in\N$. By \propref{gr-R=eRe} and \propref{gr-R=eRe'}, $R_{\infty,\la}$ and $R'_{\infty,\phi}$ are isomorphic to some direct limits of the systems $\{R_{n,\la}\}_{n\in\N}$ and $\{R'_{n,\phi}\}_{n\in\N}$, respectively. On the other hand,
$\{E_{n,\la}\}_{n\in\N}$ and $\{E'_{n,\phi}\}_{n\in\N}$ form inverse systems by \propref{ExtLonto} and \propref{ExtL'onto}. From \eqref{R-E},  we may expect that $R_{\infty,\la}$ and $R'_{\infty,\phi}$ are isomorphic to some subalgebras of the inverse limits of the systems $\{E'_{n,\phi}\}_{n\in\N}$ and $\{E_{n,\la}\}_{n\in\N}$, respectively.
\end{rem}

\end{document}